\documentclass[12pt]{amsart}
\usepackage{amsmath, amsthm, amssymb, amscd}
\usepackage[dvipsnames]{xcolor}

\usepackage[titletoc,toc, title]{appendix}

\usepackage{mathtools}

\usepackage{hyperref}
\usepackage{cleveref}
\usepackage{mabliautoref}

\usepackage{geometry}
\DeclareMathOperator{\rank}{rank}
\DeclareMathOperator{\ncr}{nc-rank}
\DeclareMathOperator{\Hom}{Hom}

\DeclareMathOperator{\Ann}{Ann}

\DeclareMathOperator{\Spec}{Spec}
\DeclareMathOperator{\Supp}{Supp}

\DeclareMathOperator{\Mat}{Mat}

\DeclareMathOperator{\Min}{Min}
\DeclareMathOperator{\Minh}{Minh}

\DeclareMathOperator{\eh}{e}
\DeclareMathOperator{\FL}{FL}
\DeclareMathOperator{\GL}{GL}

\DeclareMathOperator{\fsig}{s} 
\DeclareMathOperator{\dsig}{s_{dual}} 
\newcommand{\Zsig}[1]{\operatorname{s_{Tr}^{#1}}}

\newcommand{\gdsig}[2]{\fsig_{\mathrm{surj}}(#1 \twoheadrightarrow #2)} 
\newcommand{\basegdsig}[3]{\fsig^{#3}_{\mathrm{surj}}(#1 \twoheadrightarrow #2)} 
\newcommand{\asn}[2]{\operatorname{asn}(#1 \twoheadrightarrow #2)}

\newcommand{\red}[1]{{#1}_{\textup{red}}}

\newcommand{\cotimes}[1]{%
  \mathbin{\widehat{\mathop{\otimes}_{#1}}}%
}

\newcommand{\RR}{\mathbb{R}}
\newcommand{\CC}{\mathbb{C}}
\newcommand{\ZZ}{\mathbb{Z}}

\newcommand{\length}{\ell}

\newcommand{\Z}{\mathcal{Z}}

\newcommand{\ehk}{\eh_{HK}}
\newcommand{\mf}{\mathfrak}

\newcommand{\ul}[1]{\underline{#1}}

\newcommand{\blank}{\ul{\;\;\;}}

\newcommand{\m}{{\mathfrak m}}
\newcommand{\n}{{\mathfrak n}}
\newcommand{\p}{{\mathfrak p}}
\newcommand{\q}{{\mathfrak q}}

\newtheorem*{statement*}{Statement}
\newtheorem*{fact*}{Fact}
\numberwithin{equation}{theorem}

\begin{document}

\title[The asymptotic surjective number and  F-surjective signature]{The asymptotic surjective number \\and  F-surjective signature}

\author{Ilya Smirnov}
\address{BCAM -- Basque Center for Applied Mathematics, Mazarredo 14, 48009 Bilbao, Basque Country, Spain \quad and \quad
Ikerbasque -- Basque Foundation for Science, Plaza Euskadi 5, 48009 Bilbao, Basque Country, Spain}
\email{ismirnov@bcamath.org}

\author{Kevin Tucker}
\address{Department of Mathematics, University of Illinois at Chicago, Chicago, IL 60607, USA}
\email{kftucker@uic.edu}

\dedicatory{Dedicated to the memory of Professor Mitsuyasu Hashimoto}

\begin{abstract}
We study two invariants measuring the asymptotic abundance of surjections between modules.
We prove that Hashimoto's asymptotic surjective number is attained over local rings, hence it is a rational number. We also prove that its associated function is constructible in families and on spectra; on the latter, it is moreover lower semicontinuous on the support of the target.
We establish the existence of the $F$-surjective signature, a common generalization of $F$-signature and dual $F$-signature, and prove that it is  superadditive in its source and varies lower semicontinuously across the spectrum under mild assumptions.
As applications, we obtain new rationality results for dual $F$-signature and prove a transformation inequality under finite maps, with consequences for quotient singularities and reduction modulo $p$.
\end{abstract}

\maketitle

\section*{Introduction}
\label{intro}
In the study of singularities of positive characteristic 
various classes of singularities are defined by 
requiring the Frobenius endomorphism to have special properties. This approach originates from the work of Kunz \cite{Kunz1}
characterizing regular rings by the flatness of Frobenius.
Within this program \emph{strongly F-regular} singularities emerged as the mildest class satisfying many desirable 
properties \cite{HochsterHuneke2}. 
For example, they are normal and Cohen--Macaulay, the torsion in the class group is finite \cite{PolstraGroup, Martin}, 
and the local {\'e}tale fundamental group is finite \cite{CRST}.  

An important tool for working with this class is \emph{F-signature}, 
an invariant of local rings $(R, \m,k)$ of positive characteristic $p$ 
that \emph{quantifies} strong F-regularity:
$\fsig(R) > 0$ precisely when $R$ is strongly F-regular \cite{AberbachLeuschke} and 
$\fsig(R) = 1$ if and only if $R$ is regular \cite{HunekeLeuschke}. 
If we assume that the Frobenius endomorphism $F \colon R \to R$ is a 
finite map (e.g., if $R$ is complete and its residue field is perfect), then
the F-signature of $R$ was defined in \cite{HunekeLeuschke} as
\[
\fsig(R) = \lim_{e \to \infty} \frac{\text{the maximal rank of a free summand of $F_*^e R$}}{[k:k^{p^e}]p^{e \dim R}},
\]
where the denominator equals $\rank F_*^e R$ when $R$ is a domain.
The origin of this definition is \cite{SmithVanDenBerghSimplicityOfDiff}
where more general limits of the form
\[
     \lim_{e \to \infty} \frac{\max \{n \mid N^{\oplus n} \text{ is a direct summand of $F_*^e M$}\}}{[k: k^{p^e}]p^{e \dim R}} \tag{$\star$}
\]
were considered in a special case where $M$ has the \emph{finite Frobenius representation type} (FFRT) property -- by definition, the property requires that the set of all indecomposable direct summands of $\{F_*^e M\}_{e \geq 1}$
is finite. It is not known whether the limit in ($\star$) exists in full generality,
the FFRT case was settled in \cite{SmithVanDenBerghSimplicityOfDiff, Yao} 
and existence for quotient singularities was investigated in \cite{HashimotoSymonds, HashimotoKobayashi}. However, when $N = R$ 
the limit was shown to exist in \cite{Tucker}. 

Following Sannai's approach in \cite{Sannai} to generalizing F-signature, we consider the invariant
\[
\gdsig{M}{N} = \limsup_{e \to \infty}\frac{\max \{n \mid \text{there exists a surjection } F_*^e M \twoheadrightarrow N^{\oplus n}\} }{[k: k^{p^e}] p^{e \dim R}}, 
\]
which we will call the \emph{F-surjective signature}
of $M$ onto $N$. We note that F-signature is the F-surjective signature of $R$ onto $R$ and by taking $M = N = \omega_R$, the 
dualizing module of a Cohen-Macaulay local ring $R$, 
we obtain Sannai's \emph{dual F-signature} 
$\dsig(R) = \gdsig{\omega_R}{\omega_R}$.
This invariant quantifies F-rational singularities \cite{Sannai, SmirnovTucker}, which are a class defined as a cohomological weakening of strong F-regularity \cite{FedderWatanabe, Velez}: while strong F-regularity requires that the map $R \to R^{1/p^e}$
given by $1 \mapsto x^{1/p^e}$ is \emph{pure} (universally injective) for any nonzero $x$ and $e \gg 0$, F-rationality requires the same map to be injective on local cohomology. Building on the preceding work of Hochster and Yao \cite{HochsterYao} that pursued a different approach,
\cite{Sannai} shows that $\dsig(R) > 0$ if and only if $R$
is F-rational, while the maximal value $1$ is taken if and only if $R$ is regular. 

In \cite{SmirnovTucker}, by reconciling the Hochster--Yao and Sannai's approaches, the authors showed that the limit in $\dsig(R)$ exists. 
The first principal result of this note is the existence of F-surjective signatures, $\gdsig{M}{N}$,  in full generality. 
\begin{theorem*}[]
Let $(R, \m, k)$ be an F-finite local ring with the residue field $k$ and $M, N$ be finitely generated $R$-modules with $N \neq 0$. 
Then the sequence 
\[
\frac{\max \{n \mid \text{there exists a surjection } F_*^e M \twoheadrightarrow N^{\oplus n}\} }{[k: k^{p^e}] p^{e \dim R}}
\]
converges, i.e., $\gdsig{M}{N}$ is well-defined. We further set $\gdsig{M}{0} = \infty$.

The function $M \mapsto \gdsig{M}{N}$
is superadditive along short exact sequences. Moreover, if $R$ is a domain and not a perfect field, then $\gdsig{M}{N} = \rank (M) \gdsig{R}{N}$
for any $R$-module $M$ of full support.
\end{theorem*}

This rank formula over a domain, \emph{source independence},  was shown in the case $N = R$ in \cite{Tucker}. 
Likewise, our result has several consequences for the dual F-signature, 
but let us first discuss its proof.  
The proof builds on the results of \cite{SmirnovTucker} and utilizes an analogue of Frobenius splitting ideals for the dual F-signature. 
We show that the surjective F-signatures can be similarly recovered from 
a generalization of Frobenius splitting ideals, 
generalizing an equivalent definition of F-signature originating from \cite{Tucker, Yao, AberbachEnescu2}. The second 
ingredient is another asymptotic invariant introduced by Hashimoto in \cite{HashimotoDual} for the purpose of computing the dual F-signature of invariant subrings. 

\begin{definition*}[\cite{HashimotoDual}]\label{def: asn}
Let $R$ be a ring and $M, N$ be finitely generated $R$-modules with $N \neq 0$. 
The \emph{asymptotic surjective number} of $M$ onto $N$ is defined by setting
\[
\asn {M}{N} = 
\lim_{m \to \infty}
\frac{\max \{n \mid \text{there exists a surjection }  M^{\oplus m} \twoheadrightarrow N^{\oplus n}\}}{m}.
\]
It is often convenient to set $\asn {M}{0} = \infty$. 
\end{definition*}

Due to the superadditivity of the sequence, 
this limit is in fact the supremum.
We use the linear algebra machinery of \cite{SmirnovTucker} to obtain 
a stabilization result 
for the \emph{asymptotic surjective number} and a number of consequences 
of this fact and of the machinery developed.

\begin{theorem*}[{Theorem~\ref{thm: asn is achieved}, Propositions~\ref{prop: asn semi},~\ref{prop: asn under flat extension}}]
    Let $R$ be a Noetherian ring and $M, N$ be finitely generated $R$-modules. Then
    \begin{enumerate}
        \item if $R$ is local and $N \neq 0$, then $\asn{M}{N}$ is achieved at a finite $m$;
        \item if $R$ is local and $R \to S$ is a flat local map, 
        then $\asn{M}{N} = \asn {M \otimes_R S}{N \otimes_R S}$;
        \item if $\dim R < \infty$ and $N \neq 0$, then $\asn{M}{N}$ is a rational number;
        \item $\Spec (R) \ni \p \mapsto \asn {M_\p}{N_\p}$ is constructible and lower semicontinuous on $\Supp (N)$.
    \end{enumerate}
\end{theorem*}

Extending Hashimoto's observation, 
we express the dual F-signature as an asymptotic surjective number in rings that satisfy the FFRT property. 
In particular, the dual F-signature is a rational number 
for module-finite subrings of rings with finite Frobenius representation type (FFRT), for example, non-modular quotient singularities. Together with toric singularities, which were treated in \cite{SmirnovTucker}, 
these give essentially all cases where it is known that F-signature is rational. Second, we give a lower bound for the dual F-signature of 
module-finite subrings of regular rings. This result extends the work of Sannai \cite{Sannai} and Hashimoto \cite{HashimotoDual} where it was shown that the dual F-signature of an F-regular quotient singularity is at least $1/|G|$. Third, we give an inequality bounding dual F-signature in finite extension, correcting a proof in \cite{SmirnovTucker}.
Last, it can be used to simplify the known computations of the dual F-signature \cite{Veronese, Nakajima1, Nakajima2} by removing the need for obtaining an explicit decomposition of $F_*^e \omega_R$ as a direct sum of indecomposable modules.

The stabilization result will also be used in the forthcoming work \cite{MukhopadhyaySmirnov} to 
reprove \cite{Yao}, where it was shown that F-splitting signatures of $N$ in $M$ exist when $M$ has the FFRT property, and to show that the dual F-signature defines a constructible function on the spectrum of a ring satisfying the FFRT property. 

\subsection*{Acknowledgements}
\begin{itemize}
    \item We thank Anurag Singh for discussions and 
Junquan Qin for bringing our attention to the gap in 
the proof of \cite[Proposition~3.19]{SmirnovTucker}.
    \item The first-named author was supported by the Spanish Ministry of Science, Innovation and Universities by grants PCI2024-155055-2 funded by MICIU/AEI/10.13039/501100011033 and UE, PID2024-156181NB-C31 funded by MICIU/AEI/10.13039/501100011033 and \\ FEDER, UE, RYC2020-028976-I funded by MICIU/AEI/10.13039/501100011033 and EI ESF "ESF Investing in your future”.    
    \item The second author was partially supported by
NSF Grant \#2501904 and Simons Foundation Travel Support for Mathematicians SFI-MPS-TSM-00014083.
    \item Most of this work was completed in the pre-AI era, with an important exception of Theorem~\ref{thm: matrix blowups give injections}: the connection to \cite{DerksenMakam} was suggested by  an early version of Gemini and was shaped into a complete proof by the authors.  ChatGPT was used extensively in the final stages of manuscript preparation, to help refine arguments and examples (in particular, examples in Remarks~\ref{asn not semicontinuous}, \ref{rem: signature via asn} and Example~\ref{ex:counterexample} were suggested by ChatGPT) and improve exposition.   We take full responsibility for the mathematical content and final presentation of the manuscript.
\end{itemize}

\section{Preliminaries}

A Noetherian ring $R$ is \emph{semi-local} if it has finitely many maximal ideals. 
A \emph{local} ring $(R, \m)$ is a Noetherian commutative ring such that $\m$ is the only maximal ideal of $R$. Occasionally, we will also use notation $(R, \m, k)$ to emphasize the residue field $k = R/\m$.
For a prime ideal $\p \subset R$ we will denote  
$\kappa (\p) = R_\p/\p R_\p$. 

\subsection{Frobenius basics}

\begin{definition}
    Let $R$ be a ring of characteristic $p > 0$ and $F \colon R \to R$ 
    be the Frobenius endomorphism.
    We say that $R$ is F-finite if $F_* R$ is a finitely generated $R$-module.
\end{definition}

By \cite{Gabber} a Cohen-Macaulay F-finite ring 
has  a dualizing module. 

\begin{definition}
    Let $R$ be a reduced Noetherian ring and $M$ be a finitely generated $R$-module. 
    We say that $M$ has generic rank $r$ if
    $\dim_{\kappa(\mf p)} M_{\mf p} = r$ for every minimal prime ideal $\mf p$ such that $\dim R/\mf p = \dim R$.
\end{definition}

If $(R, \m, k)$ is an F-finite local ring \cite[Proposition~2.3]{Kunz2} 
asserts that, for any $\p \in \Spec(R)$,  $[\kappa(\p): \kappa(\p)^{p}] = p^{\dim (R/\p)} [k: k^p]$. 
In particular, if $R$ is reduced and $M$ has generic rank then 
\[
\rank (F_*^e M) = p^{e\dim (R)}[k : k^{p^e}] \rank (M).
\] 

\begin{definition}
    Let $(R, \m, k)$ be an F-finite local ring. 
    We will denote $\delta(R) \coloneqq [k : k^p]p^{\dim R}$.
\end{definition}

\subsubsection{F-rational signature}
Sannai's dual F-signature $\dsig(R) = \gdsig{\omega_R}{\omega_R}$ is defined only when $R$ is F-finite, however 
it can be extended by the (normalized) F-rational signature 
\cite{HochsterYao, SmirnovTucker}. Namely, whenever  $J$ is a parameter ideal 
in a F-finite Cohen-Macaulay local ring $(R, \m)$ we have 
by \cite[Corollary~5.6]{SmirnovTucker} that
\[
\dsig(R) = \min_{J \subsetneq I} \left\{\frac{\ehk(J) - \ehk(I)}{\length (I/J)} \right\},
\]
but the right-hand side is defined even 
if $R$ is non F-finite. 
Abusing the notation, we will use $\dsig(R)$ to denote the right-hand side even when $R$ is not F-finite. For non-Cohen-Macaulay local rings, we set $\dsig(R)=0$.
When $R$ is formally equidimensional, this agrees with the
displayed formula by \cite{SmirnovTucker}.

\subsubsection{Gamma-construction}
We will use the $\Gamma$-construction to reduce to the F-finite case, 
we refer to \cite[6.11]{HochsterHuneke2} and \cite{Lyu} for more details.
Let $k$ be a field of characteristic $p > 0$ and fix a $p$-basis $\Lambda$ of $k$. For any cofinite subset $\Gamma \subset \Lambda$, we may form a field
$k^\Gamma$ by adjoining all $p^{e}$-roots of all elements of $\Gamma$ to $k$. If $A = k[[x_1, \ldots, x_n]]$
and $A \subseteq R$ is a module-finite extension, 
then we define $\widehat{A^\Gamma} \coloneqq k^\Gamma[[x_1, \ldots, x_n]]$
and $\widehat{R^\Gamma} = R \cotimes{A} \widehat{A^\Gamma}$. 

\begin{theorem}[{\cite[Theorem~7.1]{Lyu}}]\label{thm: Lyu's gamma}
Let $(R, \m)$ be a complete local ring of characteristic $p > 0$
and $A \subseteq R$ be a complete regular local ring such that $R$ is a finitely generated $A$-module.
Then for any cofinite $\Gamma$ we have
$\dsig(\widehat{R^\Gamma}) \leq \dsig(R)$ and 
for any $\varepsilon > 0$ there exists a cofinite
$\Gamma_0 \subset \Lambda$ such that 
\[
\dsig(\widehat{R^\Gamma}) \leq \dsig(R)\leq \dsig(\widehat{R^\Gamma}) + \varepsilon
\]
for all cofinite subsets $\Gamma \subseteq \Gamma_0$. In particular, 
\(\dsig (R) = \sup_{\Gamma} \dsig(\widehat{R^\Gamma}).\)
\end{theorem}
\begin{proof}
    Since $R \to \widehat{R^\Gamma}$ is faithfully flat, $\dsig(\widehat{R^\Gamma}) \leq \dsig(R)$. The second inequality follows from the argument of \cite[Theorem~7.1]{Lyu} and the invariance of $\dsig(R)$ under the completion. 
\end{proof}

\subsection{Skew fields}
In this section, rings are not necessarily commutative.

\subsubsection{Linear algebra over skew fields}
A \emph{skew field} (or a division ring) is a nonzero ring in which every nonzero element has a multiplicative inverse. 
A vector space $V$ over a skew field $F$ is a {\bf right} $F$-module.
The dimension of $V$ can be defined as the cardinality of a basis of $V$. 

\begin{definition}
    Let $A$ be a matrix over a skew field $F$. 
    The rank of $A$ is the dimension of the image of $A$ ($A$ acts on the left on the space of columns).
\end{definition}

\begin{proposition}[{e.g., \cite[7.2]{Hungerford}}]
    Let $A$ be an $n \times m$ matrix over a skew field $F$.
    The following are equal:
    \begin{enumerate}
        \item the dimension of the image of $A$ ($A$ acts on the left on the space of columns);
        \item the number of right-independent columns of $A$; 
        \item the minimal $r$ such that $A = P Q$ where $P$ is $n \times r$ and $Q$ is $r \times m$; 
        \item the integer $r$ such that 
        $P A Q = E_r$, where $P$ and $Q$ are invertible $n \times n$ 
        and $m \times m$ matrices and $E_r$ is an $n \times m$ matrix of the form
        $E_r = \begin{bmatrix}
            I_r & 0\\
            0 & 0
        \end{bmatrix}$, 
        where $I_r$ is the $r \times r$ identity matrix. 
    \end{enumerate}
    We will call this number the rank of $A$. 
\end{proposition}
\begin{proof}
    These are verified as in the commutative case, note that the row (resp., column) operations of the Gauss algorithm are given by left (resp., right) multiplication by elementary matrices and these matrices are invertible by the construction. 
\end{proof}

\subsubsection{The skew field of rational functions}
We will now introduce the free skew field, a non-commutative analog 
of the field of rational functions. 

\begin{definition}
    If $K$ is a field we will use $K \langle x_1, \ldots, x_n\rangle$ to 
    denote the ring of non-commutative polynomials
    on indeterminate variables $x_1, \ldots, x_n$.
\end{definition}

\begin{definition}
    Let $R$ be a non-commutative ring. A skew field $R \subseteq K$
    is called a \emph{skew field of fractions} of $R$ if $K$ is generated by $R$ as a skew field. 
    
    Furthermore, $K$ is called the \emph{universal skew field of fractions} of $R$ if for any skew field $L$ and any morphism of rings $\phi \colon R \to L$ there exists a subring $R \subseteq K_0 \subseteq K$ and a morphism of rings $\theta \colon K_0 \to L$ such that 
        $\phi$ is the restriction of $\theta$ to $R$ and 
        for every $x \in K_0$ the image $\theta(x)$ is nonzero if and only if $x$ is invertible in $K_0$. 
\end{definition}

The subtlety is that already $k \langle x, y \rangle$
admits non-isomorphic field of fractions. The universal skew field of fractions is the initial object in the category of skew fields of fractions of $R$ (i.e., skew fields $R \subseteq K$
with morphisms $K \to L$ given by a ring homomorphism $\beta \colon K_0 \to L$
from a subring $R \subseteq K_0 \subseteq K$ with the property that $\beta$ is the identity on $R$ and $\ker (\beta)$ is precisely the set of non-units of $K_0$). 

\begin{definition}
    Let $K$ be a field. 
    The free skew field $K \langle \langle x_1, \ldots, x_n\rangle\rangle$
    is the universal skew field of fractions of $K \langle x_1, \ldots, x_n\rangle$. 
\end{definition}

The existence of the free skew field was first shown by Amitsur \cite{Amitsur}. We refer to \cite{KVVBackground} for the explicit construction of $K \langle \langle x_1, \ldots, x_n\rangle\rangle$ and for an overview of other approaches. 

\subsubsection{Non-commutative rank of a linear matrix}

\begin{definition}
Let $F$ be a field and let $M_0, \ldots, M_k \in \Mat_{n \times m} (F)$, 
let $t_1, \ldots, t_k$ be formal indeterminates. 
We will call $M = M_0 + M_1t_1 + \cdots + M_kt_k$ a \emph{linear matrix}
associated to $M_0, \ldots, M_k$.  

The rank of $M$, $\rank (M)$, is the rank of $M$ as a matrix over $F(t_1, \ldots, t_k)$ and the non-commutative rank of $M$, $\ncr (M)$
is the rank of $M$ over $F \langle \langle t_1, \ldots, t_k\rangle\rangle$.
\end{definition}

The next result characterizes non-commutative rank through the \emph{tensor blow-up} where for a subspace $H \subseteq \Mat_{n \times m} (F)$
the subspace $H \otimes \Mat_d (F) \subseteq \Mat_{nd \times md}(F)$
consists of block matrices with blocks in $H$.

\begin{proposition}
    Let $F$ be an infinite\footnote{Strictly speaking, a large enough finite field will do.} field. Let $H = \langle M_1, \ldots, M_k\rangle$ be a subspace of $n \times m$ matrices over $F$ 
    and let $\mathcal{M} = \sum M_i t_i$ be the linear matrix associated to $H$. 
    Then 
    \[
    \rank (\mathcal{M}) = \max \{\rank (M) \mid M \in H \}, \text{ and }
    \] 
    \begin{align*}
    \ncr (\mathcal{M})  &= 
    \lim_{d \to \infty} \frac{1}{d}\max \{\rank (M) \mid M \in H \otimes \Mat_d (F)\} \\ &= 
    \max_{d \geq 1} \{\frac{1}{d}\max \left \{\rank (M) \mid M \in H \otimes \Mat_d (F)\} \right \}.
    \end{align*}
\end{proposition}
\begin{proof}
    By taking minors it suffices to assume that $n = m$, 
    in this case these characterizations are discussed on \cite[p.3]{IQS}.  
\end{proof}

\section{The asymptotic surjective number}

The goal of this section is to study further an asymptotic invariant introduced by Hashimoto in \cite{HashimotoDual}. 
The results of this section are characteristic-free. 

\subsection{Convergence of the asymptotic surjective number}

It was shown by Hashimoto that the limit in $\asn{M}{N}$ exists by elementary means (\cite[Lemma~3.6]{HashimotoDual}), but the following proof provides an explicit convergence estimate. 

\begin{theorem}\label{thm:linalgsurjections}
Let $k$ be a field and $V, W$ be finite dimensional nonzero $k$-vector spaces. 
Fix a non-empty subspace $H \subseteq \Hom_k (V, W)$ and 
denote by $n(H)$ the maximal integer $n$ for which there is a surjection 
$V \twoheadrightarrow W^{\oplus n}$ with all components in $H$.  
Then 
\[
0 \leq
\min_{Q \neq 0} \left \{ \frac{\dim \left( V/(\cap_{h \in H} \ker (\pi_Q \circ h))\right)}{\dim Q} \right \} - n(H)
< (\dim W)^4
\]
where the minimum ranges through all nonzero quotients $\pi_Q\colon W \twoheadrightarrow Q \neq 0$.
\end{theorem}
\begin{proof}
First, if there is a surjection $V \twoheadrightarrow W^{\oplus n}$ given by $(h_1, \ldots, h_n) \in H^{\oplus n}$ for some $n \in \mathbb{Z}_{\geq 0}$, then the compositions $(\pi_Q \circ h_1, \ldots, \pi_Q \circ h_n)$
define a surjective homomorphism $V \twoheadrightarrow Q^{\oplus n}$. 
Thus
\[
\dim \left( V/(\cap_{h \in H} \ker (\pi_Q \circ h))\right)
\geq \dim \left( V/(\cap_{i = 1}^n \ker (\pi_Q \circ h_i))\right) = n \dim Q.
\]
This gives the nonnegativity. 

The opposite inequality is a consequence of \cite[Corollary~A.12]{SmirnovTucker}. Set
\[
\mu = \min_{Q \neq 0} \left \{ \frac{\dim \left( V/(\cap_{h \in H} \ker (\pi_Q \circ h))\right)}{\dim Q} \right \}.
\]
Note that the infimum is in fact a minimum, since the set consists of rational numbers with denominators bounded by \(\dim W\).

Suppose for contradiction that
\[
\mu \ge n(H)+1 + \frac{(\dim W)^4 - (\dim W)^2}{6}.
\]
Then by \cite[Corollary~A.12]{SmirnovTucker}, there exists a surjection
$V \twoheadrightarrow W^{\oplus (n(H)+1)}$
with all components in \(H\), contradicting the definition of \(n(H)\). Therefore,
\[
\mu < n(H)+1 + \frac{(\dim W)^4 - (\dim W)^2}{6}.
\]
In particular,
\(\mu < n(H) + (\dim W)^4,\)
which proves the desired inequality.
\end{proof}

This theorem is well-suited for proving the existence of various asymptotic measures of surjections. For convenience, we introduce the following notation.

\begin{definition}\label{def: Z submodule}
    Let $R$ be a ring and $M, N$ be $R$-modules. 
    For a submodule $K \subseteq N$, we define
    \[
    \Z^M (K) \coloneqq \left\{ x \in M \mid h(x) \in K \; \forall h \in \Hom_R (M, N) \right\} = \bigcap_{h \in \Hom_R (M, N)} h^{-1} (K).
    \]
\end{definition}

\begin{corollary}\label{cor: asn convergence}
Let $(R, \m, k)$ be a local ring and $M, N$ be finitely generated $R$-modules with $N$ nonzero. Then
\[
0 \leq \asn{M}{N} - \frac{\max \{n \mid \text{there exists a surjection } M^{\oplus m} \twoheadrightarrow N^{\oplus n}\}}{m}
\leq \frac{(\dim_k (N/\m N))^4}{m}.
\]

Moreover, 
\[
\asn{M}{N} = 
\min \Bigl\{ \frac{\dim_k(M/\Z^M(K))}{\dim_k(N/K)} \,\Bigm|\, \m N \subseteq K \subsetneq N \Bigr\}
\]
and, in particular, it is a rational number. 
\end{corollary}
\begin{proof}
Let \(X = M/\mathfrak m M\) and \(Y = N/\mathfrak m N\), and let 
\(H \subseteq \Hom_k(X,Y)\) be the image of the natural map 
\(\Hom_R(M,N) \to \Hom_k(X,Y)\).
If \(f \colon M^{\oplus m} \to N^{\oplus n}\) is an \(R\)-linear map such that 
\(f \otimes_R k\) is surjective, then \(f\) is surjective by Nakayama’s lemma.
Thus,  to compute the asymptotic surjective number, it suffices to consider surjections \(X^{\oplus m} \twoheadrightarrow Y^{\oplus n}\) of \(k\)-vector spaces arising from maps with component maps in \(H\).

Fix a nonzero quotient $\pi_Q\colon Y \twoheadrightarrow Q$ and consider the sequence 
\[
m \mapsto s_m (Q) \coloneqq \frac{1}{m \dim_k Q} \dim_k \left( \frac{X^{\oplus m}}{\bigcap_{h_1, \ldots, h_m \in H} \ker \left( X^{\oplus m} \xrightarrow{(h_1, \ldots, h_m)} Y \xrightarrow{\pi_Q} Q \right) }\right).
\]
By Theorem~\ref{thm:linalgsurjections}, applied with $V = X^{\oplus m}$, $W = Y$, and the maps $X^{\oplus m} \xrightarrow{(h_1, \ldots, h_m)} Y$ with $h_1, \ldots, h_m \in H$, we know that 
\begin{equation}\label{eq: asn and s}
    0 \leq \min_Q s_m(Q) -
\frac{\max \{n \mid \text{there exists a surjection } M^{\oplus m} \twoheadrightarrow N^{\oplus n} \}}{m} 
\leq \frac{(\dim_k Y)^4}{m}.
\end{equation}
 
Now let
\[
(x_1, \ldots, x_m) \in 
\bigcap_{h_1, \ldots, h_m \in H}
\ker\bigl( X^{\oplus m} \xrightarrow{(h_1, \ldots, h_m)} Y \xrightarrow{\pi_Q} Q \bigr).
\]
Fix \(1 \leq i \leq m\). By taking \(h_i = h\) and \(h_j = 0\) for all \(j \neq i\), we see that
\(h(x_i) \in \ker \pi_Q\) for all \(h \in H\).
Therefore,
\[
\bigcap_{h_1, \ldots, h_m \in H} \ker \left( X^{\oplus m} \xrightarrow{(h_1, \ldots, h_m)} Y \xrightarrow{\pi_Q} Q\right) 
= \left(\bigcap_{h \in H} \ker (\pi_Q \circ h)\right)^{\oplus m}
\]
since the previous argument shows the left-hand side is contained in the right-hand side, and the reverse inclusion is clear.
Thus $s_m(Q) = s_1(Q)$ for all $m \ge 1$, so the sequence is constant. Taking the limit as $m \to \infty$ in (\ref{eq: asn and s}) yields that
\(\asn{M}{N}  = \min_{Q \neq 0} s_1(Q).\)
It also provides the required convergence estimate after noting that 
if $Q = N/K$, then 
$s_1(Q) = \frac{\dim_k(M/\Z^M(K))}{\dim_k(N/K)}$.
\end{proof}

\subsection{The asymptotic surjective number is achieved}

We will now show that the limit in Definition~\ref{def: asn}
is achieved.

\begin{theorem}\label{thm: asn is achieved}
Let $(R, \m)$ be a local ring and $M, N$ be finitely generated $R$-modules. If $N$ is nonzero, then there is $m \geq 1$ such that 
\[
\asn{M}{N}  = \frac{\max \{n \mid \text{there exists a surjection }  M^{\oplus m} \twoheadrightarrow N^{\oplus n}\}}{m}.
\]    
\end{theorem}
\begin{proof}
It suffices to assume that $\asn{M}{N} > 0$.
Since $\asn{M}{N}$ is a rational number by Corollary~\ref{cor: asn convergence}, we may replace $M$ and $N$ by direct sums and assume that $\asn{M}{N}  = 1$. Now, the alternative characterization of
$\asn{M}{N} $ in Corollary~\ref{cor: asn convergence}, 
reduces the conclusion to a linear algebra problem: 
in the notation of the proof of Corollary~\ref{cor: asn convergence}
it suffices to demonstrate that the condition
\[
\dim_k \left( \frac{X}{\cap_{h \in H} \ker \{X \xrightarrow{h} Y \xrightarrow{\pi_Z} Z\}}\right)
\geq \dim_k Z
\]
for all $\pi_Z$ implies the existence, for some $m$, of a surjection
$X^{\oplus m} \twoheadrightarrow Y^{\oplus m}$
with all components in $H$. 
Theorem~\ref{thm: matrix blowups give injections} provides the dual form of 
this claim.     
\end{proof}

\begin{theorem}\label{thm: matrix blowups give injections}
Suppose that $X, Y$ are finite-dimensional $k$-vector spaces and $H \subseteq \Hom_k (X, Y)$ is a subspace. If
\[
\dim_k \bigg ( \sum_{h \in H} h(V) \bigg ) \geq \dim_k V \text{ for every } V \subseteq X,
\]
then there exists $m \geq 1$ and an injection $\phi_m \colon X^{\oplus m} \to Y^{\oplus m}$ in $H \otimes \Mat_{m \times m} (k)$, \emph{i.e.}, $\phi_m$ is given by a matrix with entries in $H$.
\end{theorem}
\begin{proof}
    It follows from the assumption that $y \coloneqq \dim Y \geq x \coloneqq \dim X$.
    Choose a basis $h_0, \ldots, h_s$ of $H$
    and define a matrix $M = h_0 + h_1t_1 + \cdots + h_st_s$ 
    with entries $k[t_1, \ldots, t_s]$.
    We claim that $\ncr (M) = x$. By contradiction, if $\ncr(M) < x$, then 
    \cite[Theorem~1]{FR} (based on \cite[Corollary~6.3.6]{Cohn}) shows that there are non-negative integers $i, j$ and 
    invertible matrices $U \in \GL_y(k), V \in \GL_x(k)$
    such that $i + j = \ncr(M)$ and 
    \[
    \{UhV \mid h \in H\} \subseteq \left \{\begin{bmatrix}
        A & C\\
        B & 0
    \end{bmatrix} \mid A \in \Mat_{j \times i} (k), B \in \Mat_{(y - j) \times i} (k), C \in \Mat_{j \times (x - i)} (k)
    \right \}.
    \]
    Working with the second block, we see that 
    there is $(x - i)$-dimensional subspace of $X$ mapped into a $j$-dimensional subspace of $Y$ by all $h \in H$. Hence, by the assumption we must have $x \leq j + i = \ncr (M)$, 
    a contradiction. 

    We now truncate (for example, using the Gauss algorithm) $M$ to obtain an $x \times x$ square matrix $N$ such that $\ncr(N) = \ncr(M)$. 
    Write $N = n_0 + n_1t_1 + \cdots + n_st_s$.
    By \cite[Proposition~7.3]{HW} (see also \cite[1.2]{IQS} and the proof of \cite[Proposition~2.1]{KVV}), there exists a positive integer $m$
    and matrices $Q_0, \ldots, Q_s \in \Mat_{m\times m} (k)$
    such that
    \[
    \sum_{i = 0}^s n_i \otimes Q_i \text{ is invertible.}
    \]
    Therefore $\sum_{i = 0}^s h_i \otimes Q_i \in H \otimes \Mat_{m \times m} (k)$ has full rank and the assertion follows. 
\end{proof}

\begin{remark}
An example from \cite{SmirnovTucker} illustrates the difference 
between commutative and non-commutative ranks.
In the example, $M$ and $T$ are three-dimensional vector spaces 
and $H$ is spanned by 
\[
A = \begin{bmatrix}
    1 & 0 & 0 \\
    0 & 1 & 0 \\
    0 & 0 & 0 \\    
\end{bmatrix}, 
B = \begin{bmatrix}
    0 & 0 & 1\\
    0 & 0 & 0\\
    0 & 1 & 0\\    
\end{bmatrix},
\text{ and }
C = \begin{bmatrix}
    0 & 0 & 0 \\
    0 & 0 & -1\\
    1 & 0 & 0 
\end{bmatrix}.
\]
It is easy to see that there is no injection $M \to T$ in $H$, but one can form an injection $M^{\oplus 2} \to T^{\oplus 2}$ by the matrix 
\[
\begin{bmatrix}
    A & B\\
    C & A
\end{bmatrix}.
\]
\end{remark}

\begin{remark}
    If $k$ is infinite and $\dim X > 2$, then \cite[Theorem~1.8]{DerksenMakam} can be used to show that $m < \dim X$ in Theorem~\ref{thm: matrix blowups give injections}.
    It would be interesting to know if there is an effective (not requiring rescaling $N$ and $M$ as in the present proof) bound 
    in Theorem~\ref{thm: asn is achieved} as it would help 
    computing dual F-signature.
\end{remark}

\subsection{Further properties of the asymptotic surjective number}

\begin{proposition}\label{prop: asn under flat extension}
    Let $(R, \m, k) \to (S, \mf n, \ell)$ be a faithfully flat map of local rings and $M, N$ be finitely generated $R$-modules. 
    Then 
    $\asn{M \otimes_R S}{N \otimes_R S} = \asn{M}{N}$.
\end{proposition}
\begin{proof}
    We may assume that $M, N \neq 0$ since otherwise the assertion is trivial. 
    Since surjections ascend, we immediately obtain an inequality 
    $\asn {M \otimes_R S}{N \otimes_R S} \geq \asn{M}{N}$.
    Thus it suffices to assume that $\asn {M \otimes_R S}{N \otimes_R S} > 0$ and, by replacing $M$ and $N$ with direct sums we may assume that 
    $\asn {M \otimes_R S}{N \otimes_R S} = 1$.
    
    Since $S$ is flat, $\Hom_{S} (M \otimes_R S, N \otimes_R S) = \Hom_R (M, N) \otimes_R S$. Fix minimal generators $h_0, \ldots, h_n$ of $\Hom_R (M, N)$ and denote by $\overline{h_0}, \ldots, \overline{h_n}$
    the induced maps $M \otimes_R k \to N \otimes_R k$. 
    By the proof of Theorem~\ref{thm: matrix blowups give injections} 
    the generic matrix $A = \overline{h_0} + \overline{h_1}t_1 + \cdots + \overline{h_n}t_n$ has the full rank over the skew field $\ell \langle\langle t_1, \ldots, t_n\rangle \rangle$. 
    However, by \cite[Theorem 6.4.6]{Cohn} there is an embedding of skew fields $k \langle\langle t_1, \ldots, t_n\rangle \rangle \to \ell \langle\langle t_1, \ldots, t_n\rangle \rangle$, 
    thus $A$ has full rank over $k \langle\langle t_1, \ldots, t_n\rangle \rangle$ and hence $\asn{M}{N} = 1$ by the proof of Theorem~\ref{thm: asn is achieved}.

\end{proof}

We will now use the convergence estimate to globalize our findings, showing semicontinuity of local asymptotic surjective numbers and that the asymptotic surjective number is rational in non-local rings. 

\begin{definition}
    Let $f \colon \Spec (R) \to \mathbb{R} \cup \{+ \infty\}$ be a function.
    We say that $f$
    is \emph{lower semicontinuous} if for every $a \in \mathbb{R}$
    the set $f^{-1}((a, \infty]) = \{\p \in \Spec(R) \mid f(\p) > a\}$ is open. 
    A function $f$ is \emph{constructible} if $f^{-1} (a)$ 
    is a constructible subset of $\Spec(R)$ for all $a \in \RR \cup \{+\infty\}$ 
    and $f$ takes only finitely many values.
\end{definition}

\begin{remark}
    If $\Spec (R)$ is Noetherian, then a function $f$ is constructible if and only if $f$ is \emph{generically constant}, i.e., for every $\p \in \Spec(R)$ there is an open subset $0 \neq U \subseteq V(\p)$ such that $f(\p) = f(\q)$ for all $\q \in U$ (\cite[053Z]{stacks}).
\end{remark}

\begin{remark}
If $R$ is Noetherian, then $f$ is constructible and lower semicontinuous if and only if $f^{-1}([a, \infty]) = \{\p \in \Spec(R) \mid f(\p) \geq a\}$ is open
for all $a \in \RR \cup \{+\infty\}$. Therefore, by Nagata's criterion (e.g., \cite[Proposition~2.4]{SmirnovSurvey})
this property is equivalent to showing that for $\p, \q \in \Spec(R)$:
\begin{enumerate}
    \item $f(\p) \geq f(\q)$ whenever $\p \subseteq \q$; 
    \item $f$ is generically constant. 
\end{enumerate}
\end{remark}

\begin{proposition}\label{prop: asn semi}
Let $R$ be a commutative Noetherian ring and 
$M, N$ be finitely generated $R$-modules.
Then the function 
$
\p \mapsto \asn{M_\p}{N_\p}
$
is constructible and lower semicontinuous on $\Supp (N)$.

Moreover, if $\dim (R) < \infty$ and $N \neq 0$, then 
$
\asn{M}{N} = \min \{ \asn{M_\p}{N_\p} \mid \p \in \Supp (N)\}.
$
In particular, $\asn{M}{N}$ is a rational number. 
\end{proposition}
\begin{proof}
For a prime $\p \in \Supp (N)$ let us denote by $\gamma_m (\p)$
the local \emph{surjective capacity}
\[
\gamma_m (\p) = \max \{n \mid \text{there exists a surjection }  M_\p^{\oplus m} \twoheadrightarrow N_\p^{\oplus n}\}.
\] 
It is easy to see that the function $\p \mapsto \gamma_m(\p)$ is 
constructible and lower semicontinuous on $\Supp (N)$, for example by adapting the argument of \cite[Lemma~2.2]{PolstraTucker} to verify the Nagata criterion. 
By passing to the suprema, we obtain that $\asn{M_\p}{N_\p} \geq \asn{M_\q}{N_\q}$ whenever $\p \subseteq \q$.
Furthermore, since $\asn{M_\p}{N_\p}$ is achieved, 
by spreading the optimal surjection 
we may find $s \notin \p$
such that $\asn{M_\q}{N_\q}\geq \asn{M_\p}{N_\p}$
for all $s \notin \q$. 
The first conclusion follows from Nagata's criterion.

Second, since $\Supp(N)$ is quasi-compact, a lower semicontinuous function attains its minimum (e.g., \cite[Theorem~2.3]{SmirnovTucker}).
Thus, it suffices to show that 
\[
\asn{M}{N}= \inf_{\p \in \Spec(R)} \asn{M_\p}{N_\p}.
\]
Since $N$ is finitely generated, $\sup \{ \dim (N_\p/\p N_\p) \mid \p \in \Spec (R)\} < \infty$, so Corollary~\ref{cor: asn convergence}
gives uniform convergence of $\gamma_m (\p)/m$ 
that in turn allows to exchange the limit and the infimum
\[
\inf_{\p \in \Spec(R)} \asn{M_\p}{N_\p}
= \lim_{m \to \infty} \frac{1}{m}
\inf_{\p \in \Spec (R)} \gamma_m (\p).
\]
Next, \cite[Theorem~3.9]{Baidya} implies that
\[
\inf_{\p \in \Spec(R)} \gamma_m(\p) - \dim R
\le
\max \{n \mid \text{there exists a surjection } M^{\oplus m} \twoheadrightarrow N^{\oplus n}\}
\le
\inf_{\p \in \Spec(R)} \gamma_m(\p).
\]
The conclusion follows after dividing both sides by $m$ and taking the limits.    
\end{proof}

For the next result we want to introduce some notation. 
Suppose that $B$ is a ring, $R$ be a finite type $B$-algebra, and $I$ be an ideal of $R$ such that $R/I \cong B$ as $B$-algebras. 
Note that $I$ is a maximal ideal of $R_\p/\p R_\p$ for any $\p \in \Spec (B)$.
Thus $\Spec (B)$ parametrizes a family of local rings 
$R(\p)$, defined as the localization of $R_\p/\p R_\p$ at the image of $I$.
For any $R$-module $M$ we also set $M(\p) = M \otimes_R R(\p)$. 

\begin{corollary}\label{cor: semicontinuity in families}
    Suppose that $B$ is a Noetherian ring, $R$ is a finite type $B$-algebra, and $I$ is an ideal of $R$ such that $R/I \cong B$ as $B$-algebras. 
    If $M, N$ are finitely generated $R$-modules, then 
    \[
    \Spec (B) \ni \p \mapsto \asn{M(\p)}{N(\p)}
    \]
    is a constructible function.    
\end{corollary}
\begin{proof}
    We will prove that the function is generically constant.
	By passing to $B/\p B$, we may assume that $B$ is a domain and $\p$ is its generic point. Since $B$ is a domain, there is an open set $\emptyset \neq U \subset \Spec(B)$ such that for any $\q \in U$
    \[
    \Hom_R (M, N) \otimes_B \kappa(\q) \cong 
    \Hom_{R \otimes_B \kappa(\q)} (M \otimes_B \kappa(\q), N \otimes_B \kappa(\q)),
    \]
    see, for example, \cite[(9.4.3)]{EGA43}.
    If $N(\p) = 0$, then after shrinking $U$ the modules $N(\q)$ vanish for all $\q \in U$, so the function is constantly $\infty$ on $U$. We may therefore assume $N(\p) \neq 0$, whence $N(\q) \neq 0$ for every $\q \in U$.
    Any map in $\Hom_R (M(\q), N(\q))$
    lifts to $\Hom_R (M_{(\q, I)}, N_{(\q, I)})$. 
    In particular, if $\phi\colon M(\q)^{\oplus m} \twoheadrightarrow N(\q)^{\oplus n}$ is a surjection such that $\asn{M(\q)}{N(\q)} = n/m$, then there is $\Phi \in \Hom_R (M_{(\q, I)}^{\oplus m}, N_{(\q, I)}^{\oplus n})$ such that $\phi$ is a base change of $\Phi$. 
	Applying Nakayama's lemma in $R_{(\q, I)}$, we obtain that 
	$\Phi \otimes_R R_{(\q, I)}$ is surjective. Localizing further, it follows that $\asn{M(\p)}{N(\p)} \geq n/m = \asn{M(\q)}{N(\q)}$.
	
    Conversely, take a surjection $M(\p)^{\oplus m} \twoheadrightarrow N(\p)^{\oplus n}$ such that $\asn{M(\p)}{N(\p)} = n / m$. Collecting denominators, we find $r\in R\setminus I$ such that there is a surjection 
    $(M/\p M)_r^{\oplus m} \twoheadrightarrow (N/\p N)_r^{\oplus n}$.
    After shrinking $U$ by inverting the image of $r$
    in $R/I\cong B$, base change gives
    $\asn{M(\q)}{N(\q)}\ge n/m=\asn{M(\p)}{N(\p)}$
    for every $\q\in U$.
\end{proof}

\begin{remark}\label{asn not semicontinuous}
    Semicontinuity does not hold in general because the asymptotic surjective number need not always increase or decrease under specialization. 
    Namely, take $B = k[t]$, $R = B[x]$, $I = (x)$,
    $M = R \oplus R/(x-t)$, and $N = R \oplus R/(x)$.
    Let $\p = (0)$ and $\m = (t)$. 
    Then an easy computation shows that
    \[
    \asn{M(\p)}{N(\p)} = 1/2 < 1 = \asn{M(\m)}{N(\m)}, \text{while} 
    \]
    \[
    \asn{R(\p)}{M(\p)} = 1 > 1/2 = \asn{R(\m)}{M(\m)}.
    \]

\end{remark}

\section{F-surjective signatures}
We now turn to positive characteristic and establish existence of $F$-surjective signatures. We first set a definition that presents 
an analogue of Frobenius splitting ideals \cite{Yao2, AberbachEnescu0}
and generalizes \cite[Definition~4.3]{SmirnovTucker}.

\begin{definition}\label{def Z}
Let $(R, \mf m, k)$ be a local F-finite ring of characteristic $p > 0$, and $M, N$ finitely generated $R$-modules. 
For a submodule $V \subseteq N$,  
we denote $\Z_e^M (V) \coloneqq \Z^{F_*^e M} (V)$, i.e.,  
the submodule of $F_*^e M$ given by 
\[
\Z_e^M (V) \coloneqq 
\{
x \in F_*^e M \mid \phi (x) \in V \quad \forall \phi \in \Hom_R (F_*^e M, N)
\}.
\]
\end{definition}

We now state the basic properties of these submodules.

\begin{proposition}\label{prop: basic property Z}
Consider a local $F$-finite ring $(R, \mf m, k)$ of characteristic $p > 0$ 
and finitely generated $R$-modules $M, M', N$.
Let $V \subseteq N$ be a submodule and denote by $\pi$ the quotient map $N \to N/V$.
Then 
\begin{enumerate}
\item if $\m N \subseteq V$, then $\mf m F_*^e M \subseteq \Z_e^M (V)$,
\item $\Z_e^M (V) = \cap_{\phi \in \Hom_R (F_*^e M, N)} \ker (\pi \circ \phi)$,
\item $\Z_e^M (V)$ is an $F_*^e R$-submodule of $F_*^e M$, and
\item if $\psi \colon M' \to M$ is a homomorphism of $R$-modules, 
then $F_*^e \psi (\Z_e^{M'}(V)) \subseteq \Z_e^M (V)$,  \label{part: Z is functorial}
\item $\Z_e^{M \oplus M'} (V) = \Z_e^{M} (V) \oplus \Z_e^{M'} (V)$.
\label{part: Z and direct sum}
\end{enumerate}
\end{proposition}
\begin{proof}
The first two properties are immediate from the definition. Third, 
if $x \in \Z_e^M (V)$ and $r \in R$, then for any $\phi \in \Hom_R (F_*^e M, N)$ we may view $\phi (F_*^e r x)$ as the evaluation of the composition map at $x$, so $\phi (F_*^e r x) \in \ker \pi$ and the property follows.
Property (\ref{part: Z is functorial}) follows from the definition upon noting that the composition with $F^e_*\psi$ gives a map $\Hom_R (F_*^e M, N) \to \Hom_R (F_*^e M', N)$.
Last, (\ref{part: Z and direct sum}) is immediate from the definition since $\Hom_R (F_*^e (M\oplus M'), N)$ is generated by maps vanishing on one of the summands.
\end{proof}

We use submodules $\Z_e$ to define a generalization of F-signature. 
We will later show that this generalization recovers the $F$-surjective signature. 

\begin{definition}\label{def: Z signature}
Let $(R, \mf m, k)$ be a local F-finite ring of characteristic $p > 0$ and 
$M, N$ be finitely generated $R$-modules. 
For a submodule $\m N \subseteq V \subsetneq N$ 
the corresponding generalized F-signature is defined as
\[
\Zsig{M} (V) = 
\lim_{e \to \infty} \delta(R)^{-e}\dim_k ((F_*^e M)/\Z_e^M(V)) = 
\lim_{e \to \infty} \frac{\dim_k ((F_*^e M)/\Z_e^M(V))}{[k : k^{p^e}]p^{e\dim R}}. 
\]
\end{definition}

We will now prove that this limit exists and that convergence is uniform in $V$.

\subsection{Existence and uniform convergence}

Recall that $\Minh (R) = \{\mf p \in \Spec R \mid \dim R/\mf p = \dim R\}$. 

\begin{lemma}\label{lem: comparison generically equal}
Let $(R, \mf m, k)$ be a local F-finite ring of characteristic $p > 0$ and 
suppose that $L$ and $M$ are two finitely generated $R$-modules
and $L_\mf p \cong M_\mf p$ for 
every prime $\mf p \in \Minh (R)$. 
Then there exists a constant $C$ such that for any
finitely generated $R$-module $N$,
any submodule $\m N \subseteq V \subsetneq N$, and any $e \geq 1$ one has an estimate
\[
|\dim_k ((F_*^e M)/\Z_e^M(V)) - \dim_k ((F_*^e L)/\Z_e^L(V)) | < Cp^{e (\dim R - 1)}[k^{1/p^e}: k].
\]
\end{lemma}
\begin{proof}
Let $W = R \setminus \bigcup_{\mf p \in \Minh (R)} \mf p$.  By the assumption, 
$W^{-1} M \cong W^{-1} L$, so there is an exact sequence 
$L \xrightarrow{\alpha} M \xrightarrow{\beta} T \to 0$
such that $\dim T < \dim R$. 
By Proposition~\ref{prop: basic property Z} 
we therefore have the exact sequence
\[
F_*^e L/\Z_e^L (V) \xrightarrow{\overline{F_*^e \alpha}} F_*^eM/\Z_e^M(V) \xrightarrow{\overline{F_*^e \beta}} F_*^eT/F_*^e \beta(\Z_e^M(V)) \to 0.
\]
Taking lengths, we immediately obtain that 
\[
\dim_k \frac {F_*^e M}{\Z_e^M(V)} - \dim_k \frac{ F_*^e L}{\Z_e^L(V)} 
\leq \dim_k \frac{F_*^e T}{F_*^e \beta(\Z_e^M(V))}
\leq \dim_k \frac{F_*^e T}{F_*^e \beta(\mf m F_*^e M)}
\leq \dim_k F_*^e \frac{T}{\mf m^{[p^e]}T}.
\]
Next, since the Hilbert--Kunz function is $O(p^{e \dim (T)})$, (see \cite[Lemma~3.2]{Tucker}) we have an estimate
\[
\dim_k F_*^e (T/\mf m^{[p^e]}T) =  
[k^{1/p^e}: k] \length_R \left (T/\mf m^{[p^e]}T \right)
< C_T [k^{1/p^e}: k] p^{e (\dim R - 1)}
\]
for some constant $C_T$ that depends only on $T$. 
This further allows to bound
\[
\dim_k ((F_*^e M)/\Z_e^M(V)) < \dim_k ((F_*^e L)/\Z_e^L(V))  + C_Tp^{e (\dim R - 1)}[k^{1/p^e}: k]
\]
and the assertion follows since we may switch $L$ and $M$ in the argument. 
\end{proof}

\begin{corollary} \label{cor: existence of Z signature}
Let $(R, \mf m, k)$ be an  F-finite local ring of characteristic $p > 0$ and $M$ be a finitely generated $R$-module.
Then there exists a constant $C$ such that 
for every finitely generated $R$-module $N$, 
any submodule $\m N \subseteq V \subsetneq N$, and any $e \geq 1$ 
we have 
\[
\left |\delta(R)^{-e}\dim_k ((F_*^e M)/\Z_e^M(V)) - \Zsig{M} (V) \right | < Cp^{-e}.
\]
\end{corollary}
\begin{proof}
First, assume that $M$ is an $\red{R}$ module. For every $\p \in \Minh(R)$, $M_\p$ and $F_*M_\p$ are $\kappa(\p)$-vector spaces, and counting dimensions using that $[\kappa(\p):\kappa(\p)^p] = \delta(R)$ (cf. \cite[Corollary~2.4]{PolstraTucker}) gives $(F_*M)_\p \cong M_\p^{\oplus\delta(R)}$.
By Lemma~\ref{lem: comparison generically equal} there exists a 
constant $C > 0$, independent of $N$ and $V$, such that 
\[
\left |
\delta(R)
\dim_k ((F_*^e M)/\Z_e^M(V))
- \dim_k ((F_*^{e} F_* M)/\Z_{e}^{F_* M}(V))
\right | 
< C[k^{1/p^e}: k]p^{e (\dim R - 1)}.
\]
After noting that $\Z_{e}^{F_* M}(V) = \Z_{e + 1}^M (V)$
and dividing by $\delta(R)^{e + 1} = [k^{1/p^{e+1}}: k]p^{(e+1)\dim R}$, we  bound
\[
\left |
\frac{\dim_k ((F_*^e M)/\Z_e^M(V))}{\delta(R)^{e}}
- \frac{\dim_k ((F_*^{e + 1} M)/\Z_{e + 1}^M(V))}{\delta(R)^{e + 1}}
\right | 
< \frac{C}{\delta(R) p^e}.
\]
Summing up these inequalities (see  \cite[Lemma~3.5]{PolstraTucker}) yields that for any $e' \geq 1$
\[
\left |
\frac{\dim_k ((F_*^e M)/\Z_e^M(V))}{\delta(R)^{e}}
- \frac{\dim_k ((F_*^{e +e'} M)/\Z_{e + e'}^M(V))}{\delta(R)^{e + e'}}
\right | 
< \frac{C}{\delta(R)} \sum_{i = 0}^{e' - 1} \frac{1}{p^{e + i}}
< \frac{2C}{\delta(R)} p^{-e} 
\]
and the assertion follows in this case after taking the limit as $e' \to \infty$.

In general, choose $e_0$ such that $F^{e_0}_*M$ is an $\red{R}$-module. From the argument above and dividing by $\delta(R)^{e_0}$, we obtain 
that the bound 
\begin{equation}\label{eq: bound in the proof of uniform convergence}
    \left |
\frac{\dim_k ((F_*^{e_0 + e} M)/\Z_{e + e_0}^M(V))}{\delta(R)^{e_0 + e}}
- \frac{\dim_k ((F_*^{e_0 + e +e'} M)/\Z_{e_0 + e + e'}^M(V))}{\delta(R)^{e_0 + e + e'}}
\right | 
<  \frac{2C}{\delta(R)^{e_0+1}p^e}
= \frac{2Cp^{e_0}}{\delta(R)^{e_0+1}}p^{-(e+e_0)}
\end{equation}
still holds for all $e,e' \geq 1$. Taking the limit as $e' \to \infty$ gives the assertion for exponents at least $e_0+1$, after adjusting $C$. The remaining finitely many exponents can be accommodated uniformly in $N$ and $V$,
since
\(
0 \leq \dim_k ((F_*^eM)/\Z_e^M(V))
\leq \dim_k ((F_*^eM)/\m F_*^eM).
\)
\end{proof}

A standard consequence of uniform convergence is the ability to interchange limits.

\begin{corollary}\label{cor: change limits in Z signature}
Let $(R, \mf m, k)$ be a local F-finite ring of characteristic $p > 0$ and $M$ be a finitely generated $R$-module. 
Then there exists a constant $C$ such that 
for any finitely generated $R$-module $N \neq 0$ and any $e \geq 1$
\[
\left | \delta(R)^{-e} \inf_V \{ \dim_k ((F_*^e M)/\Z_e^M(V)) \} - \inf_V \{\Zsig{M} (V)\} \right | < Cp^{-e},
\]
where both infima range over all submodules $\m N \subseteq V \subsetneq N$.
\end{corollary}

\subsection{Properties of F-surjective signatures}

We start studying properties of F-surjective signatures with a trivial vanishing result for F-surjective signature 
and its computation in the simplest case.

\begin{lemma}\label{lem: trivial vanishing of signature}
    Let $R$ be an F-finite Noetherian ring and $M, N$ be finitely generated $R$-modules with $N \neq 0$.
    If $\sqrt{\Ann (M)} N \neq 0$, then there are no surjections
    $F_*^e M \twoheadrightarrow N$ for $e \gg 0$.
    In particular, if $R$ is local and $\gdsig{M}{N} > 0$,
    then $\sqrt{\Ann (M)} N = 0$.
\end{lemma}
\begin{proof}
    If  $\sqrt{\Ann (M)} N \neq 0$, then 
    there are no surjections $F_*^e M \twoheadrightarrow N$
    for $e \gg 0$, since $F_*^e M$ are eventually $R/\sqrt{\Ann (M)}$-modules.
\end{proof}

\begin{lemma}\label{lem: F-surjective signature for artinian}
    Let $(R, \m, k)$ be an Artinian local F-finite ring of characteristic $p > 0$ and $M, N$ be finitely generated $R$-modules with $N$ nonzero. 
    Then the sequence 
    \[
    e \mapsto \max \{n \mid \text{there exists a surjection } F_*^e M \twoheadrightarrow N^{\oplus n}\}/[k:k^{p^e}] 
    \]
    converges, and 
    \[
    \gdsig{M}{N} =
    \begin{cases}
        \left\lfloor \dfrac{\length_R(M)}{\dim_k N} \right\rfloor
            & \text{if } \m N = 0 \text{ and } k = k^p,\\[10pt]
        \dfrac{\length_R(M)}{\dim_k N}
            & \text{if } \m N = 0 \text{ and } k \neq k^p,\\[6pt]
        0   & \text{otherwise}.
    \end{cases}
    \]
    If $k$ is perfect, the sequence is eventually constant.
\end{lemma}
\begin{proof}
By Lemma~\ref{lem: trivial vanishing of signature} 
we may assume $\m N = 0$. Now choose $e_0$ such that $\m^{[p^{e_0}]} = 0$, so that $\m F^e_*M = 0$ and $\dim_k F^e_*M = [k:k^{p^e}] \length_R(M)$ for all $e \geq e_0$. Then
\[
    \max\{n \mid \text{there exists a surjection }
    F_*^e M \twoheadrightarrow N^{\oplus n}\}
    =
    \left\lfloor
    \frac{[k:k^{p^e}]\length_R(M)}{\dim_k N}
    \right\rfloor.
    \]
    Dividing by $[k:k^{p^e}]$ gives the remaining conclusions.
\end{proof}

Extending one of the main results  in \cite{SmirnovTucker}, we will now show that the generalization of F-signature defined through $\Z_e$ coincides with the surjective F-signature. 
This comes through the linear algebra machinery obtained in \cite{SmirnovTucker} and will 
imply that the surjective F-signature is often \emph{source-independent}, i.e., it depends only on the rank of the source module. Our main ingredient is a result connecting two signatures.

\begin{corollary}\label{cor: Z formula for surjections}
Let $(R, \mf m, k)$ be a local F-finite ring of characteristic $p > 0$ and 
$M, N$ be finitely generated $R$-modules with $N \neq 0$.
Denote $b_e = \max \{n \mid \text{there exists a surjection } F_*^e M \twoheadrightarrow N^{\oplus n}\}$.
Then 
\[
b_e + (\dim_k N/\mf mN)^4 \geq 
\inf \left \{ \frac{\dim_k ((F_*^e M)/\Z_e^M(V))}{\dim_k (N/V)} \mid \mf m N \subseteq V \subsetneq N \right \} \geq 
b_e.
\]
\end{corollary}
\begin{proof}
Similar to Corollary~\ref{cor: asn convergence} 
we let $X = k \otimes_R F_*^e M$, $Y = N / \mf m N$, and $H \subseteq \Hom_k (X, Y)$ consists of the maps induced by $\Hom_R (F_*^e M, N)$ so that
\[
b_e = \max \{n \mid \text{there exists a surjection } X \to Y^{\oplus n} \text{ with all components in } H\}.
\]
Now, the map $V \mapsto N/V$ is a bijection between submodules $V$ of $N$ 
in the statement and nonzero quotients $\pi_Q \colon Y \twoheadrightarrow Q$
such that $F_*^e M/\Z_e^M(V) = X/\bigcap_{h \in H} \ker (\pi_Q \circ h)$. The conclusion thus follows from Theorem~\ref{thm:linalgsurjections}.
\end{proof}

\begin{theorem}\label{thm: surjections through F-signature}
Let $(R, \mf m, k)$ be a local F-finite ring of characteristic $p > 0$ and $M, N$ be finitely generated $R$-modules with $N$ nonzero. 
Then there exists a constant $C$ and a real number $\gdsig{M}{N}$ such that 
\[
\left | \gdsig{M}{N} -
\delta(R)^{-e}\max \{n \mid \text{there exists a surjection } F_*^e M \twoheadrightarrow N^{\oplus n}\}
\right | < Cp^{-e}.
\]
Moreover, if $\delta(R) > 1$, then
\[
\gdsig{M}{N} = 
\inf \left \{ \frac{\Zsig{M} (V)}{\dim_k (N/V)} \mid \mf m N \subseteq V \subsetneq N\right \}.
\]
\end{theorem}
\begin{proof}
    The first claim follows easily from Lemma~\ref{lem: F-surjective signature for artinian} if $\delta(R) \coloneqq [k: k^{p}]p^{\dim R} = 1$.  
    Otherwise, the claim follows from Corollaries~\ref{cor: Z formula for surjections}, \ref{cor: change limits in Z signature}, 
    since $(\dim_k N/\m N)^4/\delta(R)^e \to 0$. 
\end{proof}

\begin{corollary}\label{cor: determined by minimal primes}
Let $(R, \mf m)$ be a local F-finite ring of characteristic $p > 0$ and $M, M', N$ be finitely generated $R$-modules.
If $M_\mf p \cong M'_\mf p$ for all $\mf p \in \Minh (R)$, then $\gdsig{M}{N} = \gdsig{M'}{N}$.  
\end{corollary}
\begin{proof}
The claim is trivial if $N = 0$ or if $R$ is Artinian. 
In the remaining cases, the conclusion follows from Theorem~\ref{thm: surjections through F-signature}, Lemma~\ref{lem: comparison generically equal}, and 
Corollary~\ref{cor: existence of Z signature}. 
\end{proof}

\begin{corollary}\label{cor: rank formula}
Let $(R, \mf m)$ be a local F-finite ring of characteristic $p > 0$ with $\delta(R) > 1$ and $M, N$ be finitely generated $R$-modules. 
Then for any positive integer $r$ we have 
\[
\gdsig{M^{\oplus r}}{N} = r \gdsig{M}{N}.
\]
Moreover, if $R$ is reduced, $N \neq 0$, and $M$ has generic rank $r$, then $\gdsig{M}{N} = r \gdsig{R}{N}$. 
\end{corollary}
\begin{proof}
    We may assume that $N \neq 0$, otherwise the assertion is trivial. 
    By Proposition~\ref{prop: basic property Z}(\ref{part: Z and direct sum}),  $\Zsig{M^{\oplus r}} (V) = r \Zsig{M} (V)$ for all $\m N \subseteq V \subset N$.
    Thus $\gdsig{M^{\oplus r}}{N} = r \gdsig{M}{N}$ 
    by Theorem~\ref{thm: surjections through F-signature}.

    The second claim now follows from Corollary~\ref{cor: determined by minimal primes}, since 
    $\gdsig{M}{N} = \gdsig{R^{\oplus r}}{N} = r \gdsig{R}{N}$. 
\end{proof}

\begin{remark}
    We note that the asymptotic surjective number is not source-independent. Namely, if $R = k[[x^3, x^2y, y^2x, y^3]]$, then, as observed in \cite[Example~3.17]{Sannai}, $\asn {(x^2, xy, y^2)}{(x,y)} = 0$ where $(x^2, xy, y^2)$ and $(x, y)$ are $R$-submodules of $k[[x,y]]$.
\end{remark}

\begin{corollary}\label{cor: superadditive}
    Let $(R, \mf m)$ be a local F-finite ring of characteristic $p > 0$ and 
    \[0 \to L \to M \to N \to 0\] 
    be an exact sequence of finitely generated $R$-modules.
    Then for any finitely generated $R$-module $X$ we have 
    $\gdsig{M}{X} \geq \gdsig{L}{X} + \gdsig{N}{X}$.

    Moreover, the equality holds if $X$ is cyclic.
\end{corollary}
\begin{proof}
    For large $e$,  $F_*^e M$ and $F_*^e  (L\oplus N)$ are isomorphic at the minimal primes since the localized sequence splits. Thus, $\gdsig{M}{X} = \gdsig{L \oplus N}{X}$
    by Corollary~\ref{cor: determined by minimal primes}.
    The inequality $\gdsig{L \oplus N}{X} \geq \gdsig{L}{X} + \gdsig{N}{X}$ is obvious by working with surjections. 

    For the second claim,
     assume that $X \neq 0$ is cyclic. If $\delta(R) = 1$, then $R$ is Artinian with perfect residue field. By \autoref{lem: F-surjective signature for artinian}, the signature onto $X$ equals the length of the source when $X \cong k$, and vanishes otherwise, so that equality follows. If $\delta(R)>1$, then by  
    \autoref{thm: surjections through F-signature} and 
\autoref{prop: basic property Z}(\ref{part: Z and direct sum})
    \[
    \gdsig{L \oplus N}{X} = \Zsig{L \oplus N} (\m X)
    = \Zsig{L} (\m X) + \Zsig{N} (\m X)
    = \gdsig{L}{X} + \gdsig{N}{X}. \qedhere
    \]
\end{proof}

\begin{remark}
    The additivity does not always hold. 
    For example, if $R = \mathbb{F}_p[[x,y]]/(xy)$, $L = R/(x)$, and $N = R/(y)$, then $\Hom_R (N, L) = \Hom_R (L, N) = 0$ and it follows easily that $\gdsig{N}{N \oplus L} = \gdsig{L}{N \oplus L} = 0$, while $\gdsig{N\oplus L}{N \oplus L} = 1$. 
\end{remark}

The following result is inspired by \cite[Corollary~4.5]{HashimotoDual}. 

\begin{corollary}\label{cor: same support transfer}
    Let $(R, \mf m)$ be a local F-finite ring of characteristic $p > 0$ with $\delta(R) > 1$ and $M, N, X$ be finitely generated $R$-modules.
    If $\Supp (M) \subseteq \Supp (N)$ 
    and $\gdsig{M}{X} > 0$ then $\gdsig{N}{X} > 0$.  
\end{corollary}
\begin{proof}
    By \autoref{lem: trivial vanishing of signature} 
    $X$ is an $\red{R}$-module, so, after replacing $M$ and $N$ with high enough pushforwards, we may assume that $R$ is reduced. 
    Fix a positive integer $n$ such that for all $\p \in \Minh(R)$
    \[
    n \dim_{\kappa(\p)} N \otimes_R \kappa(\p) \geq \dim_{\kappa(\p)} M \otimes_R \kappa(\p).
    \]
    Then there exists a finitely generated $R$-module $M'$
    such that $(N^{\oplus n})_{\p} \cong (M \oplus M')_{\p}$ for all $\p \in \Minh(R)$.
    Hence, 
    \[
    n \gdsig{N}{X} = \gdsig{N^{\oplus n}}{X}
    = \gdsig{M\oplus M'}{X}
    \geq \gdsig{M}{X} + \gdsig{M'}{X} > 0.
    \]
\end{proof}

\begin{remark}\label{rem: signature via asn}
    If $\delta(R) > 1$, then it follows from Theorem~\ref{thm: surjections through F-signature} and Corollary~\ref{cor: asn convergence} 
    that 
    \[
    \gdsig{M}{N}  = \lim_{e \to \infty} \frac{\asn {F_*^e M} {N}}{\delta(R)^{e}}.
    \]
    On the other hand, if $R = k$ is a perfect field then
    \[
    \gdsig{k}{k^{\oplus 2}} = 0 \neq \frac 12 = \asn{k}{k^{\oplus 2}} = \inf \left \{ \frac{\Zsig{k} (V)}{\dim_k (k^{\oplus 2}/V)} \mid V \subsetneq k^{\oplus 2} \right \}.
    \]

    One could argue that $ \lim_{e \to \infty} \delta(R)^{-e}\asn {F_*^e M} {N}$ should be taken as the definition of $\gdsig{M}{N}$, since 
    it would have removed the restriction $\delta(R) > 1$
    in several preceding results. 
\end{remark}

\section{Applications}

\subsection{Module-finite subrings of a regular ring}

\begin{corollary}\label{cor: generalized for finite}
    Let $(R, \m,k)$ be an F-finite reduced equidimensional local ring of characteristic $p > 0$.
    Suppose that $R \subseteq S$ is a module-finite inclusion with $S$ regular and semi-local, and assume that $S$ is torsion-free of constant generic rank $n$ over $R$. 
    If $0 \neq M$ is a finitely generated $R$-module of constant generic rank $r$, 
    then
    $\gdsig{M}{M} = \frac{r}{n}\asn {S}{M} $. In particular, it is a rational number. 
\end{corollary}
\begin{proof}
    By \cite[Corollary~2.4]{PolstraTucker}, applied to
$R/P$ for $P\in\Min(R)$, and invariance of $p$-degree
under finite field extensions, $F_*^e S$ has constant
generic rank $\delta(R)^e$ over $S$.
    Since $S$ is regular and semi-local, it follows that $F^e_* S \cong S^{\oplus \delta(R)^e}$. If $\delta(R) > 1$, then
    \[
        \gdsig{S}{M} = \lim_{e \to \infty} \delta(R)^{-e}\max\{b\mid \mbox{ there exists  a surjection }S^{\oplus\delta(R)^e}
                         \twoheadrightarrow M^{\oplus b}\}= \asn{S}{M}.
    \]
    By Corollary~\ref{cor: rank formula}, 
    \[
    \gdsig{S}{M}=n\gdsig{R}{M}
    \qquad \mbox{ and } \qquad
    \gdsig{M}{M}=r\gdsig{R}{M},
    \]
    which proves the desired formula. If instead $\delta(R) = 1$, then $R$ is a perfect field, and the result follows after computing $\asn{S}{M}=n/r$ and $\gdsig{M}{M}=1$.
\end{proof}

\begin{corollary}\label{cor: dsignature for finite subring}
    Let $(R,\m)$ be a Cohen-Macaulay local ring of characteristic $p>0$ with canonical module $\omega_R$, and let $R\subseteq S$ be a module-finite inclusion. Suppose that either:
    \begin{enumerate}
        \item \label{firstpartcordsigfinitesubring} $R$ is $F$-finite and reduced, $S$ is regular and semilocal, and $S$ is torsion-free of constant generic rank $n$ over $R$; or,
        \item 
        \label{secondpartcordsigfinitesubring}
        $R$ is a domain and $S$ is regular and local, with $n = [S:R]$.
    \end{enumerate}
    Then
    \[
        \dsig(R) = \frac{\asn{S}{\omega_R}}{n}.
    \]
    In particular, $\dsig(R)$ is rational.
\end{corollary}
\begin{proof}
    If $R$ is $F$-finite, as in \autoref{firstpartcordsigfinitesubring} or as a special case of \autoref{secondpartcordsigfinitesubring}, 
    since $\omega_R$ has constant generic rank one, the formula follows from \autoref{cor: generalized for finite}. 

    So assume we are in case \autoref{secondpartcordsigfinitesubring} but $R$ is not $F$-finite. By Proposition~\ref{prop: asn under flat extension} we may assume that $R$ is complete. 
    Apply a common $\Gamma$-construction to $R\subseteq S$ over a complete regular Cohen subring of $R$. For sufficiently small cofinite $\Gamma$, the resulting completed rings are $F$-finite, $\widehat{S^\Gamma}$ is regular, and canonical modules and the extension rank $n$ are preserved. Thus, the $F$-finite case, Theorem~\ref{thm: Lyu's gamma}, and Proposition~\ref{prop: asn under flat extension} give
    \[
        \dsig(R)
        =\sup_\Gamma\dsig(\widehat{R^\Gamma})
        =\sup_\Gamma
          \frac{\asn{\widehat{S^\Gamma}}
                 {\omega_{\widehat{R^\Gamma}}}}{n}
        =\frac{\asn{S}{\omega_R}}{n}.
        \qedhere
    \]
\end{proof}

As an application we obtain a class of examples where dual F-signature is constructible and lower semicontinuous. 
While the dual F-signature is known 
to be lower semicontinuous under mild assumptions \cite{SmirnovTucker, Lyu}, 
it is not expected to be constructible even in hypersurfaces \cite{ASS}. 

\begin{corollary}
    Let $R \subseteq S$ be a module-finite inclusion of F-finite domains of characteristic $p > 0$. 
    If $S$ is regular, then 
    $\p \mapsto \dsig(R_\p)$ is constructible and lower semicontinuous.
\end{corollary}
\begin{proof}
    Since $R$ is excellent, its non-Cohen-Macaulay locus 
    is closed and $\dsig(R_\p) = 0$ on it by \cite{SmirnovTucker}. Thus it suffices to assume that $R$ is 
    Cohen-Macaulay, so it has a dualizing module by \cite{Gabber}.
    Since $R$ is a domain, we may now apply Corollary~\ref{cor: dsignature for finite subring} and the claim follows from 
    Proposition~\ref{prop: asn semi}.
\end{proof}

\subsection{Transformation inequalities}

In an attempt to generalize Hashimoto's work in \cite{HashimotoDual}, \cite[Proposition~3.19]{SmirnovTucker} claimed that in a module-finite extension $(R,\m,k) \subseteq (S,\n,\ell)$ of $F$-rational rings
\begin{equation}
    \label{eq:finitetransfineq}
    \dsig (R) \geq \dsig(S) \frac{[\ell : k]}{[S : R]}.
\end{equation}
Unfortunately, the example below demonstrates that this inequality need not hold when the residue field extension $k \subseteq \ell$ is not separable. 

\begin{example}
    \label{ex:counterexample}
    Let $p = 2$, $k = \mathbb{F}_2(a,b)$, and consider
    $A = k[x,y,z]/(x^2 + ay^2 + bz^2)$.
    By \cite[Theorem 3.1]{QuinlanGallegoSimpsonSingh}, 
    its second Veronese subring $B = A^{(2)}$ is $F$-rational. Set $R = B_{(x,y,z) \cap B}$, $\ell = k(a^{1/2},b^{1/2})$, and $S = \ell[y,z]_{(y,z)}$. We have finite local inclusions $R \subseteq A_{(x,y,z)} \subseteq S$ induced by $x \mapsto a^{1/2}y + b^{1/2}z$. Now, $[\ell[y,z]:k[y,z]]=4$, hence $[S:A] = 2$ and thus $[S:R] = [\ell:k] = 4$. By \cite[Theorem 2.7]{WatanabeYoshida} and since $S$ is regular, we have $\ehk(I,R) = \ehk(IS,S) = \length_S(S/IS)$ for all finite colength ideals $I \subseteq R$. Consider the parameter ideal $Q = (y^2,z^2) \subseteq R$ and $\m = (y^2,z^2,yz,xy,xz) \subseteq R$, noting that $\length_R(\m/Q) = 3$. As $\m S = (y^2,z^2,yz) \subseteq S$, we compute
    \[
    \dsig(R) \leq \frac{\ehk(Q,R) - \ehk(\m,R)}{\length_R(\m/Q)} = \frac{4-3}{3} = \frac{1}{3} < 1 = \frac{[\ell:k]}{[S:R]} \dsig(S).
    \]
    
\end{example}

We will now provide a weaker version of (\ref{eq:finitetransfineq}).

\begin{proposition}
    \label{prop: sep transf ineq}
    Let $(R,\m,k) \subseteq (S,\n,\ell)$ be a module-finite local extension of excellent local domains of characteristic $p > 0$. If $R$ is $F$-rational, then
    \[
    \dsig(R)\geq \frac{[\ell:k]_{\mathrm{sep}}}{[S:R]}\dsig(S) \geq \frac{1}{[S:R]}\dsig(S). 
    \] 
\end{proposition}

\begin{proof}
    Without loss of generality assume $\dsig(S) > 0$, so that $S$ is also $F$-rational. By \cite[Proposition 2.1, Corollary 3.7]{SmirnovTucker}, we may further assume that $R$ and $S$ are complete.

    Let $k' \subseteq \ell$ be the separable closure of $k$ in $\ell$, so that $[k':k] = [\ell:k]_{\mathrm{sep}}$. Write $k' = k(\overline{\theta})$ with minimal polynomial $\overline{g}(T) \in k[T]$. Applying Hensel's lemma to a monic lift $g(T) \in R[T] \subseteq S[T]$ of $\overline{g}(T)$ gives a root $\theta \in S$ lifting $\overline{\theta} \in \ell$. Then $R' = R[\theta] \subseteq S$ is a finite \'etale local extension of $R$ with residue field $k'$ and $[R':R] = [\ell:k]_{\mathrm{sep}}$. In particular, $R'$ is again $F$-rational, and
    \[
    \frac{1}{[S:R']} = \frac{[R':R]}{[S:R]} = \frac{[\ell:k]_{\mathrm{sep}}}{[S:R]}.
    \]

    Suppose $\ul{x}$ is a system of parameters of $R'$, hence also of $S$. Since $(\ul{x}) S \cap R' \subseteq (\ul{x})^* = (\ul{x})$ by \cite[(5.23)]{HochsterHunekeFinite}, for any $(\ul{x}) \subsetneq I$ in $R'$ we have $I / (\ul{x}) \subseteq IS / (\ul{x})S$. Thus, by \cite[Proposition~3.12]{SmirnovTucker} and \cite[Theorem~2.7]{WatanabeYoshida}, we have
    \[
    \begin{array}{rcl}
         \dsig(R) \geq \dsig(R') &=& \displaystyle \inf_{(\ul{x}) \subsetneq I} \frac{\ehk((\ul{x})) - \ehk(I)}{\length_{R'}(I / (\ul{x}))} \\ & \geq & \displaystyle
         \inf_{(\ul{x}) \subsetneq I} \frac{[\ell:k']}{[S:R']} \cdot \frac{\ehk((\ul{x})S) - \ehk(IS)}{\length_{R'}(IS / (\ul{x})S)} \\
         & = & \displaystyle
         \inf_{(\ul{x}) \subsetneq I} \frac{1}{[S:R']} \cdot \frac{\ehk((\ul{x})S) - \ehk(IS)}{\length_{S}(IS / (\ul{x})S)} \\
         & \geq & \displaystyle \frac{1}{[S:R']} \dsig(S) = \frac{[\ell:k]_{\mathrm{sep}}}{[S:R]} \dsig(S)
    \end{array}
    \]
    where the final inequality follows from the infimum in the definition of $\dsig(S)$ as there can be ideals in $S$ which are not extended from an ideal in $R'$. The last assertion follows as well from $[\ell:k]_{\mathrm{sep}} \geq 1$.
\end{proof}

\begin{remark}
We can give an alternate proof of \autoref{prop: sep transf ineq}
using the asymptotic surjective number.
Assume that $R$ is F-finite and $\dsig(S)>0$.
The case $\delta(R)=1$ is immediate, so assume $\delta(R)>1$. Since $\omega_S$ has $R$-rank $[S:R]$ and $\omega_R$
has rank one, \autoref{cor: rank formula} gives
the equality below
\[
\dsig(R)
= \frac{\basegdsig{\omega_S}{\omega_R}{R}}{[S:R]}
\geq \frac{\asn{\omega_S}{\omega_R}}{[S:R]}\dsig(S),
\]
while composing surjections and applying
\autoref{rem: signature via asn} gives the inequality.
Here we also use $\delta(R)=\delta(S)$, and the asymptotic surjective number is taken over $R$. The resulting lower bound on $\dsig(R)$ extends to the
non-$F$-finite case, after completion or assuming the existence of canonical modules, by a common $\Gamma$-construction
and \cite[Theorem~7.1]{Lyu}.

To conclude, it suffices to observe that $\asn{\omega_S}{\omega_R}\geq[\ell:k]_{\mathrm{sep}}$. We may
pass to completions and take $R'$ as in the proof of \autoref{prop: sep transf ineq} above. 
Since $R'$ is $F$-rational, $QS\cap R'=Q$ for every
parameter ideal $Q$ of $R'$. Thus the natural map on top
local cohomology from $R'$ to $S$ is injective, and local
duality gives a surjection
$\omega_S\twoheadrightarrow\omega_{R'}$. As $R'$ is finite \'etale over $R$ of rank
$[\ell:k]_{\mathrm{sep}}$, we also have
$\omega_{R'} \cong R'\otimes_R\omega_R
\cong \omega_R^{\oplus[\ell:k]_{\mathrm{sep}}}$
as $R$-modules.
Thus $\omega_S$ surjects onto $[\ell:k]_{\mathrm{sep}}$
copies of $\omega_R$ as an $R$-module, giving
$\asn{\omega_S}{\omega_R}\geq[\ell:k]_{\mathrm{sep}}$. 
\end{remark}

\begin{corollary}\label{cor: Hashimoto's inequality}
    Let $S$ be a regular semi-local domain of characteristic $p > 0$. Suppose that $(R, \m)$ is a Cohen--Macaulay subring of $S$ such that $S$ is a finite $R$-module.
    Then the following are equivalent:
    \begin{enumerate}
        \item $R$ is F-rational,
        \item $\dsig(R) \geq 1/[S : R]$.
    \end{enumerate}    
\end{corollary}
\begin{proof}
By \cite{SmirnovTucker}, $R$ is F-rational if and only if $\dsig(R) > 0$. Hence one direction is immediate. 

For the opposite direction, we assume that $R$ is F-rational and apply the argument of 
Proposition~\ref{prop: sep transf ineq}
to obtain that $I/(\ul{x}) \subseteq IS / (\ul{x})S$. 
Let $\n_1, \ldots, \n_t$ be maximal ideals of $S$.
Since every $S_{\n_i}$ is regular,  
for an $\m$-primary ideal $I$ of $R$ we have 
\[
\length_R (S/I^{[p^e]}S)
= \sum_{i = 1}^t p^{e \dim (R)}[\kappa(\n_i):k] \length_{S_{\n_i}} (S_{\n_i}/IS_{\n_i}) 
= p^{e \dim (R)} \length (S/IS).
\]
Thus $[S: R]\ehk (I) = \ehk (I; S) = \length_R (S/IS)$
and we may compute
\[
 \dsig(R) = \displaystyle \inf_{(\ul{x}) \subsetneq I} \frac{\ehk((\ul{x})) - \ehk(I)}{\length_{R}(I / (\ul{x}))} = 
 \displaystyle \inf_{(\ul{x}) \subsetneq I} \frac{\length_{R}(IS/ (\ul{x})S)}{[S:R]\length_{R}(I / (\ul{x}))} \geq \frac{1}{[S:R]}. \qedhere
\]
\end{proof}

\begin{remark}
The inequality in Corollary~\ref{cor: Hashimoto's inequality} is often strict. For example, 
the dual F-signature of the $n$th Veronese subrings of $k[[x,y]]$
is $1/2$ for all $n \geq 2$ by \cite{Sannai} (see also \cite{Matsushita, Veronese}), while $[S : R] = n$. 
\end{remark}

\begin{remark}
Suppose that $R = S^G$ is an invariant subring, where $S = k[[ x_1,\ldots,x_d ]]$ and $G$ is a finite group acting faithfully and linearly without pseudo-reflections. In this case, 
Corollary~\ref{cor: Hashimoto's inequality} 
recovers the Cohen--Macaulay case of \cite[Theorem~5.4]{HashimotoDual}, where it was shown in the modular case (i.e., $p$ divides $|G|$) that  if $\dsig (R) > 0$ then $\dsig (R) \geq 1/|G|$. In the non-modular case, the quotient singularity is strongly F-regular, so the conclusion follows  from the transformation rule for F-signature \cite{CRST, CarvajalRojas}:
\[
\dsig(R) \geq \fsig (R) = 1/[S : R] = 1/|G|.
\]
\end{remark}

\subsection{Finite Frobenius representation type}

Let $(R, \m, k)$ be a complete F-finite local ring. 
We recall that an $R$-module $M$ satisfies the FFRT property 
if the set of indecomposable $R$-modules $X$
such that $X$ is a direct summand of $F_*^e M$ for some $e$ is finite.
We will call this set the \emph{FFRT system} of $M$. 
Yao (\cite[Theorem~C]{Yao}) showed 
that, for any $X$ in the FFRT system, the \emph{F-splitting signature}
\[
\lim_{e \to \infty} \frac{\max \{n \mid X^{\oplus n} \text{ is a direct summand of } F_*^e M\}}{[k:k^{p^e}]p^{e \dim (M)}} 
\]
exists and is a rational number. 

\begin{proposition}\label{prop: FFRT formula}
    Suppose that $(R, \m, k)$ is a complete F-finite local ring
    such that $\delta(R) > 1$ and $M$ 
    is a finitely generated faithful $R$-module that has finite Frobenius representation type.
    Let $M_1, \ldots, M_n$ be the FFRT system of $M$ 
    with F-splitting signatures $s_1, \ldots, s_n$ respectively. 
    If $D$ is the common denominator of $s_i$, then
    for any nonzero finitely generated $R$-module $N$
    \[
    \gdsig{M}{N} = D^{-1} \asn {\bigoplus_{i = 1}^n M_i^{\oplus Ds_i}}{N}. 
    \]
    In particular, if $R$ is a Cohen-Macaulay reduced ring and $M$ has generic rank then $\dsig(R)$ is a rational number. 
\end{proposition}
\begin{proof}
    Write 
    $F_*^e M = \bigoplus_{i = 1}^n M_i^{\oplus a_{i,e}}$ where $a_{i, e}$ are maximal possible and let $L_i = M_i^{\oplus Ds_i}$. 
    It follows by definition that 
    \(
    \Z_e^M (V) 
    = \Z^{F_*^e M} (V) \cong 
    \bigoplus_{i = 1}^n \Z^{M_i}(V)^{\oplus a_{i,e}}.
    \)
    Thus
    \begin{align*}
    \Zsig{M} (V) &= \sum_{i = 1}^n \dim_k (M_i/\Z^{M_i}(V))
    \lim_{e \to \infty} \frac{a_{i, e}}{[k:k^{p^e}]p^{e \dim (R)}}
    = \sum_{i = 1}^n s_i \dim_k (M_i/\Z^{M_i}(V))
    \\ &= D^{-1} \sum_{i = 1}^n \dim_k (L_i/\Z^{L_i}(V))
    = D^{-1} \dim_k (\oplus L_i/\Z^{\oplus L_i}(V)). 
    \end{align*}
    By Corollary~\ref{cor: asn convergence} and Theorem~\ref{thm: surjections through F-signature} we obtain the first conclusion:
    \[
    \gdsig{M}{N} = 
    \inf \left \{ \frac{\Zsig{M} (V)}{\dim_k (N/V)} \mid\m N \subseteq V \subsetneq N \right \}
    = D^{-1}\asn{\oplus L_i}{N}.
    \]    
    The second conclusion follows from the first and Corollary~\ref{cor: rank formula}.
\end{proof}

\begin{remark}
    Related results were obtained using the \emph{Frobenius limits} of \cite{HashimotoSymonds}, which generalize the FFRT property. In the notation of the proof of Proposition~\ref{prop: FFRT formula}
    we have $\FL ([M]) = \sum^n_{i = 1} s_i [M_i] = \frac 1 D [\oplus L_i]
    $
    in the additive Grothendieck group.
    Then a slight generalization of \cite[Theorem~4.4]{HashimotoDual}
    will show that $\gdsig{M}{\omega_R} = \frac 1 D \asn {\oplus L_i}{\omega_R}$.
\end{remark}

\subsection{Reduction mod $p$}
The next consequence of the rank inequality 
is a result on the behavior of reductions mod $p$. We will start by describing our setup. 

Let $(R, \m, k)$ be a local ring of essentially finite type over $\CC$. Choose a transcendence basis
$\overline{t}_1,\ldots,\overline{t}_r$ of $k$ over $\CC$ together with lifts $t_1,\ldots,t_r\in R$. This identifies $F = \CC(t_1,\ldots,t_r) \subseteq R$ with $F \cong \CC(\overline{t}_1,\ldots,\overline{t}_r)\subseteq k$, so that $k$ is finite separable over $F$. Set $R' = (R \otimes_F k)_{\m'}$ for $\m' = \ker \{R \otimes_F k \xrightarrow{r \otimes \alpha \mapsto (r + \m)\alpha} k \}$.
The map $R \to R'$ is faithfully flat local with $\m R' = \m'R'$ and unchanged residue field $k$. Moreover, 
$R'$ is essentially of finite type over $k$, so we may write $R' = k[x_1,\ldots,x_n]_{(x_1,\ldots,x_n)} / I$. By adjoining the coefficients of a finite set of generators of $I$, we obtain a finitely generated $\ZZ$-subalgebra $C \subseteq k$ 
and a finite type $C$-algebra $A_R = C[x_1,\ldots,x_n] / J$ such that $R' = (A_R \otimes_C k)_{(x_1,\ldots,x_n)}$. We call $(C,A_R)$ a \emph{model} of $R'$. For a point $\q \in \Spec(C)$, let $R(\q) = (A_R \otimes_C \kappa(\q))_{(x_1,\ldots,x_n)}$. When $\q$ is closed, these rings have positive characteristic and will be called the \emph{geometric reductions modulo $p$} associated to the chosen model.

For a module-finite local inclusion $(R,\m) \subseteq (S,\n)$, put $S' = S \otimes_R R'$. Then $S'$ is finite and local over $R'$, since its closed fiber over $R'$ is $S / \m S$. Enlarging $C$ and replacing $A_R$ by its localization at some $s \in A_R$ whose image in $A_R / (x_1,\ldots,x_n) \cong C$ is a unit if necessary, we obtain a finite inclusion $A_R\subseteq A_S$ with
$S'\cong A_S\otimes_{A_R}R'$.
We call $(C,A_R\subseteq A_S)$ a \emph{simultaneous model}
of $R'\subseteq S'$. Set $S(\q) = A_S \otimes_{A_R} R(\q)$.
Thus, $S(\q)$ is a finite semilocal $R(\q)$-algebra.

If $R$ is Cohen-Macaulay (resp., regular), then so is $R'$, of the same dimension $d = \dim R$. Enlarging $C$ and localizing $A_R$ as above, we may assume that $A_R \otimes_C k$ is Cohen-Macaulay (resp., regular) of pure dimension $d$. By \cite[(2.3.14)(a), (2.3.5)(c), (2.3.6)(a)]
{HochsterHunekeChar0}, after further localizing $C$,
the fibers $A_R\otimes_C\kappa(\q)$ are
Cohen-Macaulay (resp., regular) of pure dimension $d$
for every $\q\in\Spec(C)$. We may also assume that $C$ is regular and that $A_R$ is flat over $C$ by generic freeness. In particular, the local rings $R(\q)$ are Cohen-Macaulay (resp., regular) of dimension $d$.
Likewise, if $S$ is regular, we may arrange that
$S(\q)$ is regular for every closed point $\q$
in a dense open subset of $\Spec(C)$.

\begin{corollary}\label{cor: reduction mod p with asn}
Let $(R, \m) \subseteq (S, \mf n)$ be a module-finite  
local $\CC$-algebra extension between rings of essentially finite type over $\CC$.
Suppose that $S$ is regular and $R$ is Cohen--Macaulay. Choose simultaneous models $A_R \subseteq A_S$ over
a finitely generated $\ZZ$-algebra $C$, as above,
and write $R(\q)$ and $S(\q)$ for the corresponding
reductions. Then there is a dense open subset $U \subseteq \Spec(C)$ such that 
\[
\dsig(R(\q)) = \frac{\asn{S}{\omega_R}}{[S:R]} 
\]
for every closed point $\q \in U$. Additionally, $R$ is normal if and only if $\asn{S}{\omega_R} > 0$, in which case
\[
\dsig(R(\q)) \geq \frac{1}{[S:R]} > 0.
\]

\end{corollary}
\begin{proof}
    By preceding spreading arguments,
    after first shrinking $\Spec(C)$ by inverting elements to get a dense open subset $U$, the chosen models give module-finite inclusions $R(\q) \subseteq S(\q)$, with $R(\q)$ reduced and Cohen-Macaulay, and $S(\q)$ regular semi-local and torsion-free of constant generic rank $[S:R]$ over 
     $R(\q)$ for all $\q \in U$. 
     As in our setup for reduction mod $p$ above, we may assume that $C$ is regular and $A_R$ is flat over $C$ with Cohen-Macaulay fibers of pure dimension $\dim(R)$. The relative canonical module $\omega_{A_R/C}$ therefore specializes, after localization, to $\omega_{R(\q)}$, e.g., by \cite[Theorem~3.6.1]{ConradBook}.

    By Corollary~\ref{cor: semicontinuity in families}, after shrinking $U$ further, we may assume that $\q \mapsto \asn{S(\q)}{\omega_{R(\q)}}$ is constant and equal to its value at the generic fiber. Since Proposition~\ref{prop: asn under flat extension} identifies this generic value with $\asn{S}{\omega_R}$, Corollary~\ref{cor: dsignature for finite subring} gives that
    \[
        \dsig(R(\q)) = \frac{\asn{S(\q)}{\omega_{R(\q)}}}{[S:R]} = \frac{\asn{S}{\omega_R}}{[S:R]}
    \]
    for every closed point $\q \in U$.

    If $R$ is normal, then the inclusion $R \subseteq S$ splits. Thus,  
    applying $\Hom_R(\blank,\omega_R)$ gives a surjection $S \cong \omega_S \cong \Hom_R(S,\omega_R) \twoheadrightarrow \omega_R$. It follows that $\asn{S}{\omega_R} \geq 1$, giving the lower bound.

    Conversely, suppose that $\asn{S}{\omega_R} > 0$, so that there is a surjection $\phi \colon S^{\oplus m} \to \omega_R$ for some $m \in \ZZ_{> 0}$. Let $K=\operatorname{Frac}(R)$ and $\overline{R}=K\cap S$, the normalization of $R$, and identify $\phi$ with its extension after applying $\blank \otimes_R K$. For any $x \in \overline{R}$, we have
    \[
    x \omega_R = x \phi(S^{\oplus m}) = \phi(xS^{\oplus m}) \subseteq \phi(S^{\oplus m}) = \omega_R,
    \]
    so that $x \in \Hom_R(\omega_R,\omega_R) = R$, showing normality.
\end{proof}

{

\begin{corollary}
Let $(R,\m) \subseteq (S,\n)$ be a module-finite local extension of normal essentially of finite type $\CC$-algebras. Choose simultaneous models $A_R \subseteq A_S$ over a finitely generated $\ZZ$-algebra $C$, as above.

Suppose that there exists $\varepsilon > 0$ and a dense open subset $U \subseteq \Spec(C)$ such that
\[
    \dsig\bigl(S(\q)_{\n_{\q,i}}\bigr)\geq\varepsilon
\]
for every closed point $\q \in U$ and every maximal ideal $\n_{\q,i}$ of $S(\q)$ (e.g. $S$ is regular with $\varepsilon=1$). Then, after further shrinking $U$, we have
\[
    \dsig(R(\q))\geq\frac{\varepsilon}{[S:R]}>0
\]
for every closed point $\q\in U$. In particular, if $R / \m = \CC$ and we may take our models with $C = \ZZ[1/N]$, then $S(p)$ is local for all sufficiently large primes $p$ and
\[
\liminf_{p \to \infty} \dsig \bigl(S(p)\bigr) > 0 \qquad \mbox{ implies } \qquad \liminf_{p \to \infty} \dsig \bigl(R(p)\bigr) > 0.
\]
\end{corollary}
\begin{proof}
    Put $n = [S:R] = [S':R']$ as in the setup for reduction mod $p$ above. We have that $S'$ has $F$-rational type, hence rational singularities by \cite{Smith2}. Since $R'$ is normal, $R' \subseteq S'$ is split by the normalized trace, so $R'$ also has rational singularities by \cite{Boutot}. Thus, after possibly shrinking $U$, the rings $R(\q)$ are $F$-rational for all closed $\q \in U$ by \cite{Hara}. Shrinking $U$ further, we may also assume that $R(\q)\subseteq S(\q)$ is a module-finite inclusion with $S(\q)$ torsion-free of rank $n$ over $R(\q)$.

    Fix a closed point $\q \in U$, and suppose $\n_{\q,1},\ldots,\n_{\q,t}$ are the maximal ideals of $S(\q)$. Then
    \[
    \widehat{S(\q)} \cong S(\q) \otimes_{R(\q)} \widehat{R(\q)} \cong \prod_{i=1}^t \widehat{S(\q)_{\n_{\q,i}}}.
    \]
    As the rings $R(\q)$ and $S(\q)_{\n_{\q,i}}$ are excellent and $F$-rational, so too are their completions. Hence each factor $\widehat{S(\q)_{\n_{\q,i}}}$ in the product above is a module-finite extension domain of $\widehat{R(\q)}$, say of degree $n_i$. Since $n = \sum_{i=1}^t n_i$, we have $n_i \leq n$ for any fixed $1 \leq i \leq t$.  Using
    \autoref{prop: sep transf ineq} and invariance under completion \cite[Corollary~3.7]{SmirnovTucker} gives
    \[
    \dsig(R(\q))
    =\dsig(\widehat{R(\q)}) \geq\frac{\dsig(\widehat{S(\q)_{\n_{\q,i}}})}{n_i}
    = \frac{\dsig(S(\q)_{\n_{\q,i}})}{n_i}
    \geq\frac{\varepsilon}{n}.
\]
For the final assertion, $S/\m S$ is an Artinian
local $\CC$-algebra with residue field $\CC$.
Spreading out its nilpotent maximal ideal shows
that $S(p)$ is local for all sufficiently large $p$.
Choosing
$0<\varepsilon<\liminf_{p\to\infty}\dsig(S(p))$,
the above bound gives the desired conclusion,
since shrinking $\Spec(\ZZ[1/N])$ excludes only
finitely many primes.
\end{proof}
}

The corresponding result for F-signature follows from \cite[Corollary~4.13]{Tucker}. This was recently improved by 
Takagi and Yamaguchi \cite[Theorem~1.3]{TakagiYamaguchi} who assume that $R \to S$ is pure and $S$ is of klt type, but not necessarily module-finite. We do not know if the corresponding statement is true for dual F-signature.

\subsection{Global properties}
We will now generalize \cite[Theorem~5.10]{SmirnovTucker} starting with a necessary uniform convergence lemma. 

\begin{lemma}\label{lem: global uniform convergence}
	Let $R$ be a Noetherian locally equidimensional F-finite ring and $M, N$ be finitely generated $R$-modules. 
	Then there exists a constant $C \geq 0$ such that for every $\p \in \Supp (N)$ 
	\[
	\left |[\kappa(\p)^{1/p^e}: \kappa(\p)]^{-1} p^{-e\dim (R_\p)} \max \{n \mid F_*^e M_\p \twoheadrightarrow  N_\p^{\oplus n}\}  -  \gdsig{M_\p}{N_\p} \right | \leq Cp^{-e}. 
	\]
\end{lemma}
\begin{proof}
	It suffices to prove the statement on each connected component of $\Spec (R)$.
	By \cite[Corollary~2.7]{Kunz2} the number $\delta \coloneqq [\kappa(\p)^{1/p}: \kappa(\p)]p^{\dim (R_\p)}$ is independent of $\p$ within one connected component. 
	If $\delta = 1$, the assertion follows from Lemma~\ref{lem: F-surjective signature for artinian}. Hence we assume that $\delta > 1$. 
	
 	Since $N$ is finitely generated, the minimal local number of generators function 
    \[
    \p \mapsto \dim_{\kappa(\p)} (N \otimes_R \kappa(\p)),
    \] 
    where $\kappa(\p)$ is the residue field of $\p$, 
    is bounded.
    It follows from Corollary~\ref{cor: Z formula for surjections}, that the convergence rate of $\delta^{-e} b_e(\p)$ is as good as the convergence rate of the function
    \[
    \p \mapsto \delta^{-e}\inf \left \{ \frac{\dim_{\kappa(\p)} ((F_*^e M_\p)/\Z_e^{M_\p}(V))}{\dim_{\kappa(\p)} N_\p/V} \mid \p N_\p \subseteq V \subsetneq N_\p \right \}. 
    \]
    It remains to explain that convergence in Corollary~\ref{cor: change limits in Z signature} can be made uniform on $\Spec (R)$. 

    First, we note that the constant $C$ given by Lemma~\ref{lem: comparison generically equal} can be independent of  $\p \in \Spec (R)$
    provided that we start with a global exact sequence 
    $L \to M \to T \to 0$ and apply
    \cite[Proposition~3.3]{Polstra}
    to bound $C_{T_\p}$ uniformly in the proof of Lemma~\ref{lem: comparison generically equal}. We obtain a suitable sequence by 
    taking $e_0$ such that $F_*^{e_0} M$ is an $\red{R}$-module, then 
    there is an exact sequence 
    \[
    F_*^{e_0 + 1} M \to F_*^{e_0} M^{\oplus \delta} \to T \to 0
    \]
    suitable for Lemma~\ref{lem: comparison generically equal}.
    The uniform constant $C$ then can be plugged 
    in (\ref{eq: bound in the proof of uniform convergence}) 
    to deduce a uniform version of Corollary~\ref{cor: existence of Z signature}: there exists a constant $C'$ such that 
    for any $\p \in \Spec (R)$, any submodule $\p N_\p \subseteq V \subsetneq N_\p$, 
    and all $e \geq e_0 + 1$
    we have 
    \[
    \left |
    \delta^{-e}\dim_{\kappa(\p)} ((F_*^{e} M_\p)/\Z_{e}^{M_\p}(V)) - \Zsig{M_\p} (V)
     \right | 
    < C' p^{-e}.
    \]
    The conclusion follows.
\end{proof}

\begin{theorem}\label{thm: semicontinuity of signature}
    Let $R$ be a Noetherian F-finite ring 
    and $M, N$ be finitely generated $R$-modules.
    If $\Supp (M)$ is locally equidimensional, 
    then the function
    $
    \p \mapsto \gdsig{M_\p}{N_\p}
    $
    is lower semicontinuous on $\Spec(R)$.    
\end{theorem}
\begin{proof}
    For brevity, let us define
    $b_e (\p) \coloneqq 
    \max \{n \mid F_*^e M_\p \twoheadrightarrow  N_\p^{\oplus n}\}$. 
    As in the proof of \cite[Theorem~5.10]{SmirnovTucker}, we will use that, for any fixed $e$, $b_e (\p)$ defines a constructible lower semicontinuous function on $\Spec (R)$.
    For $P \in \Min(R)$ put $d_P = [\kappa(P)^{1/p}:\kappa(P)]$.
    By \cite[Corollary~2.5]{PolstraTucker}, applied to $R/P$,
    \[
        \delta(R_\p)=
        \max\{d_P \mid P\in\Min(R),\ P\subseteq\p\}.
    \]
    It follows that $\p \mapsto \delta(R_\p)$ is upper
semicontinuous. Consequently,
$\p \mapsto \delta(R_\p)^{-e}b_e(\p)$ is lower
semicontinuous on $\Supp(N)$, and hence on $\Spec(R)$
(with value $\infty$ outside $\Supp(N)$).

    The case $M=0$ is immediate. For every fixed $e \geq 0$,
    \[
    \gdsig{F^e_*M_\p}{N_\p} = \delta(R_\p)^{e} \gdsig{M_\p}{N_\p}.
    \]
    Since $\delta(R_\p)^{-e}$ is positive and lower
semicontinuous, it suffices to prove the assertion
after a fixed Frobenius pushforward of $M$.

    Let $J = \sqrt{\Ann(M)}$, $\overline{R} = R/J$, and $\overline{N} = N/JN$. By replacing $M$ with $F^{e_0}_*M$ for some fixed $e_0 \gg 0$, we may assume that $M$ is an $\overline{R}$-module.
    By the argument of Lemma~\ref{lem: trivial vanishing of signature}, $b_e(\p) = 0$ for $\p \in \Supp (JN)$ and 
    outside of this set we have 
    \[
    b_e (\p) 
    = \max \{n \mid F_*^e M_\p \twoheadrightarrow  \overline{N}_\p^{\oplus n}\}.
    \]
    The assertion is trivial when $N = 0$, so, at the end, 
    it suffices to show that the convergence 
    \[
    \delta(R_\p)^{-e} b_e(\p) \to \gdsig{M_\p}{N_\p}
    \]
    is uniform for $\p \in V \coloneqq \Supp (N) \setminus \Supp (JN)$.

    Take $\p \in V$.
    We may apply Lemma~\ref{lem: global uniform convergence}
    to find a constant $C \geq 0$ such that 
	\[
	\left |\delta(\overline{R}_\p)^{-e} b_e(\p)  -  \basegdsig{M_\p}{\overline{N}_\p}{\overline{R}} \right | \leq Cp^{-e}. 
	\]    
    The constant $C$ gives uniform convergence 
    whenever $\dim (\overline{R}_\p) = \dim (R_\p)$. 
    Otherwise, by definition, $\gdsig{M_\p}{\overline{N}_\p} = 0$, while, as in the proof of Lemma~\ref{lem: global uniform convergence}, we may bound uniformly
    \[
    \frac{b_e(\p)}{\delta(R_\p)^e}
    \leq \frac{b_e(\p)}{p^{e} \delta(\overline{R}_\p)^{e}}
    \leq \frac{\dim_{\kappa(\p)} (F_*^e M \otimes_{\overline{R}} \kappa(\p))}{p^{e} \delta(\overline{R}_\p)^{e}}
    \leq \frac {C_{M}} {p^e},
    \]
    thus providing a uniform convergence estimate on the entire $V$.    
\end{proof}

As a consequence of uniform convergence and \cite[Theorem~1.1]{Baidya} 
we obtain existence of a ``global'' surjective F-signature, 
generalizing \cite{SmirnovTucker, DSPY}. 
The definition that we give uses, as in the proof of Lemma~\ref{lem: global uniform convergence} that 
$\delta(R) \coloneqq [\kappa(\p)^{1/p}: \kappa(\p)]p^{\dim (R_\p)}$ is independent of the choice of $\p$. 
If $R$ is a domain, then $\delta(R)^e = \rank F_*^e R$.

\begin{corollary}
    Let $R$ be a Noetherian locally equidimensional F-finite ring such that $\Spec(R)$ is connected. 
    Let $M, N$ be finitely generated $R$-modules with $N$ nonzero.
    Then
    the limit
    \[
    \gdsig{M}{N} = \lim_{e \to \infty} \frac{\max \{n \mid \text{there exists a surjection } F_*^e M \twoheadrightarrow N^{\oplus n}\} }{\delta(R)^e}
    \]
    exists and is equal to $\min \{\gdsig{M_\p}{N_\p} \mid \p \in \Spec (R)\}$.
    
    In particular, if $M'$ is a finitely generated $R$-module such that $M_\p \cong M'_\p$ for every $\p \in \Min(R)$, then $\gdsig{M}{N} = \gdsig{M'}{N}$. 
\end{corollary}
\begin{proof}
    An easy adaptation of \cite[Theorem~5.11]{SmirnovTucker}. The last assertion follows by applying Corollary~\ref{cor: determined by minimal primes}
over each $R_\p$ and then taking the minimum.
\end{proof}

\bibliographystyle{alpha}
\bibliography{over}

\end{document}